\newif\ifreport
\reporttrue

\documentclass[10pt]{article}

\RequirePackage{amsthm}
\RequirePackage{amsmath}
\RequirePackage{amsfonts}

\RequirePackage{amssymb}
\RequirePackage{hyperref}
\RequirePackage{multirow}
\RequirePackage{algorithm}
\RequirePackage{algpseudocode}
\RequirePackage{graphicx}
\RequirePackage{xcolor}
\allowdisplaybreaks

\usepackage{dpr_fec,ifthen,dsfont}
\usepackage{wrapfig}
\usepackage[figuesright]{rotating}

\newtheorem{assumption}{Assumption}[section]
\newtheorem{lemma}{Lemma}[section]
\newtheorem{theorem}{Theorem}[section]

\usepackage{isetechreport}
\def\reportyear{20T}
\def\reportno{015}
\def\revisionno{0}
\def\originaldate{July 29, 2020}
\def\revisiondate{\today}
\coraltrue
\cvcrfalse

\begin{document}

\title{A Subspace Acceleration Method for Minimization Involving a Group Sparsity-Inducing Regularizer}

\author{Frank E.~Curtis\thanks{E-mail: frank.e.curtis@gmail.com}}
\author{Yutong Dai\thanks{E-mail: yud319@lehigh.edu}}
\author{Daniel P.~Robinson\thanks{E-mail: daniel.p.robinson@gmail.com}}
\affil{Department of Industrial and Systems Engineering, Lehigh University}
\titlepage

\maketitle

\begin{abstract}
  We consider the problem of minimizing an objective function that is the sum of a convex function and a group sparsity-inducing regularizer. Problems that integrate such regularizers arise in modern machine learning applications, often for the purpose of obtaining models that are easier to interpret and that have higher predictive accuracy. We present a new method for solving such problems that utilize subspace acceleration, domain decomposition, and support identification. Our analysis shows, under common assumptions, that the iterate sequence  generated by our framework is globally convergent, converges to an $\epsilon$-approximate solution in at most $O(\epsilon^{-(1+p)})$ (respectively, $O(\epsilon^{-(2+p)})$) iterations for all $\epsilon$ bounded above and large enough (respectively, all $\epsilon$ bounded above) where $p > 0$ is an algorithm parameter,  
and exhibits superlinear local convergence. Preliminary numerical results for the task of binary classification based on regularized logistic regression show that our approach is efficient and robust, with the ability to outperform a state-of-the-art method.

\end{abstract}

\newcommand{\sub}[1]{^{\null}_{#1}}
\newcommand{\spgk}{s_k}
\newcommand{\GroupFaRSA}{{\tt FaRSA-Group}}
\newcommand{\LIBLINEAR}{LIBLINEAR}
\newcommand{\newGLMNET}{newGLMNET}
\newcommand{\Ik}{{\Ical_k}}
\newcommand{\Gcali}{{\Gcal_i}}
\newcommand{\Icgk}{\Ical^{\mathrm{cg}}_k}
\newcommand{\Ipgk}{\Ical^{\mathrm{pg}}_k}
\newcommand{\chicgk}{\chi^{\mathrm{cg}}_k}
\newcommand{\chipgk}{\chi^{\mathrm{pg}}_k}
\newcommand{\Icgkbar}{\bar\Ical^{\mathrm{cg}}_k}
\newcommand{\Ipgkbar}{\bar\Ical^{\mathrm{pg}}_k}
\newcommand{\chicgkbar}{\bar\chi^{\mathrm{cg}}_k}
\newcommand{\chipgkbar}{\bar\chi^{\mathrm{pg}}_k}
\newcommand{\Ismallk}{\Ical^{\mathrm{small}}_k}
\newcommand{\setepsilon}{\Kcal_\epsilon}
\newcommand{\setcg}{\Kcal^{\mathrm{cg}}}
\newcommand{\setcgZERO}{\Kcal_0^{\mathrm{cg}}}
\newcommand{\setcgSD}{\Kcal_{\mathrm{sd}}^{\mathrm{cg}}}
\newcommand{\setcgSDbig}{\Kcal_{\mathrm{sd,big}}^{\mathrm{cg}}}
\newcommand{\setcgSDsmall}{\Kcal_{\mathrm{sd,small}}^{\mathrm{cg}}}
\newcommand{\setpg}{\Kcal^{\mathrm{pg}}}
\newcommand{\setpgSAME}{\Kcal_{\!\rightarrow}^{\mathrm{pg}}}
\newcommand{\setpgDEC}{\Kcal_{\!\downarrow}^{\mathrm{pg}}}
\newcommand{\setcgpg}{\Kcal^{\mathrm{cg}\rightarrow \mathrm{pg}}}
\newcommand{\flagpgk}{{\rm flag}^{\mathrm{pg}}_k}
\newcommand{\flagcgk}{{\rm flag}^{\mathrm{cg}}_k}
\newcommand{\setalphasame}{\Kcal^\alpha_{\!\rightarrow}} 
\newcommand{\setalphadecr}{\Kcal^\alpha_{\!\downarrow}} 
\newcommand{\lb}{F_{\min}}
\newcommand{\indicator}{\mathbbm{1}}
\newcommand{\rhobar}{\bar\rho}
\newcommand{\zb}{{\tt new\_zero}}
\newcommand{\sd}{{\tt suff\_descent}}
\newcommand{\samealpha}{{\tt same\_\alpha}}
\newcommand{\decalpha}{{\tt decrease\_\alpha}}
\newcommand{\lammin}{\lambda_{\mathrm{min}}}
\newcommand{\lammax}{\lambda_{\mathrm{max}}}
\newcommand{\mumin}{\mu_{\mathrm{min}}}
\newcommand{\mumax}{\mu_{\mathrm{max}}}
\newcommand{\rhokmin}{\rho_{k,\mathrm{min}}}
\newcommand{\rhoki}{\rho_{k,i}}
\newcommand{\rhobarki}{\rhobar_{k,i}}
\newcommand{\npgk}{n_{\mathrm{pg}}(\kbar)}
\newcommand{\nbigk}{n_{\mathrm{big}}(\kbar)}
\newcommand{\nsmallk}{n_{\mathrm{small}}(\kbar)}
\newcommand{\gglasso}{{\tt gglasso}}

\section{Introduction}\label{sec.introduction}

We consider the minimization of a function that may be written as the sum of a convex function and a nonoverlapping group sparsity-inducing regularizer. Specifically, given a convex and twice continuously differentiable function $f:\R{n} \to \R{}$, a collection of $n_\Gcal > 0$ nonoverlapping groups $\Gcal:= \{\Gcal_{i}\}_{i=1}^{n_\Gcal}$ that forms a partition of $\{1,2,.\cdots,n\}$ (i.e., $\Gcal_i\cap\Gcal_j = \emptyset$ for all $i\neq j$ and $\cup_{i=1}^{n_\Gcal} \Gcal_i = \{1,2,.\cdots,n\}$), and group-wise weighting parameters $\{\lambda_{i}\}_{i=1}^{n_\Gcal} >0$, our algorithm solves the problem
\begin{equation}\label{prob}
\min_{x\in\R{n}} \{ f(x) + r(x) \}, \ \text{where} \  r(x):= \sum_{i=1}^{n_\Gcal} \lambda_i\left\|[x]_{\Gcal_{i}}\right\|_{2} 
\end{equation}
and $[x]_\Gcali$ is the subvector of $x$ corresponding to elements in $\Gcali$.  The  regularizer $r$ generalizes the $\ell_1$-norm, which is recovered by choosing $\Gcal_i = \{i\}$ for all  $i\in\{1,2,\dots,n\}$.

Despite the successes of $\ell_1$-norm regularization, its inadequacy in the context of many modern machine learning applications has been noticed by researchers, and is one motivation for the use of group regularization. In some machine learning applications the covariates come in groups (e.g., genes that regulate hormone levels in microarray data~\cite{ma2007supervised}), in which case one may wish to select them jointly. Also, integrating group information into the
modeling process can improve both the interpretability and accuracy~\cite{zeng2016overlapping} of the resulting model. Yuan and Lin~\cite{yuan2006model} observed that in the multi-factor analysis-of-variance problem, where each factor is expressed through a set of dummy variables, deleting an irrelevant factor is equivalent to deleting a \emph{group} of dummy variables; the $\ell_1$-norm regularizer fails to achieve this goal. 

\subsection{State-of-the art methods}

There is a long history of algorithms for solving regularized problems of the form~\eqref{prob} (see~\cite{bach2012optimization} and the references therein).  Here, we review some of the state-of-the-art approaches for solving sparsity-promoting problems that are most closely related to our proposed approach.

\textbf{First-order methods.}  Proximal methods are designed to solve problems of the form~\eqref{prob} and have received attention in the machine learning community~\cite{beck2009fast,combettes2011proximal,wright2009sparse}. 
A well-known example for $\ell_1$-norm regularized problems is the iterative shrinkage-thresholding algorithm (ISTA), which is obtained by applying a proximal gradient (PG) iteration to minimize a smooth function plus the $\ell_1$-norm regularizer \cite{DaubDefrDeMo04,Dono95}. Under certain assumptions, one can prove a worst-case complexity bound on the number of iterations required by the PG method before it correctly identifies the support of the optimal solution~\cite{nutini2019active}. Combined with the acceleration technique proposed by Nesterov \cite{nesterov1983method, nesterov2013gradient}, one obtains the algorithm FISTA~\cite{beck2009fast}.  One obtains a related, but distinct approach from ISTA by posing an equivalent smooth reformulation of the problem---separating the positive and negative parts of the variables---and applying a gradient projection method to the resulting formulation~\cite{FanChanHsieWangLin08,FiguNowaWrig07}. All of these approaches have been shown to work well in practice, at least 
compared to other first-order methods such as the subgradient algorithm.  However, these algorithms are often inferior in practice compared to alternative approaches that employ space decomposition techniques and/or second-order derivatives~\cite{chen2017reduced,CheCR18,OBA}.  

As an alternative to PG and gradient projection techniques, researchers have considered (block) coordinate descent for solving $\ell_1$-norm regularized problems.  Such a strategy is appealing, since when minimizing an $\ell_1$-norm regularized objective along coordinate directions, it is common that the objective is minimized with variables being zero.  These approaches are also easy to implement to exploit parallel computing; see, e.g., the accelerated randomized proximal coordinate gradient method in~\cite{xiao}, the parallel coordinate descent methods in~\cite{Richtarik2016}, and the asynchronous coordinate descent technique in \cite{liu2015asynchronous}.  A downside of these approaches is that the space decomposition is performed in a prescribed manner, rather than in an adaptive way that can benefit from information acquired during the solution process. Also, these approaches do not effectively exploit second-order derivative information and require exact minimization along coordinate directions.  An exception to this latter criticism is the inexact coordinate descent algorithm from \cite{tappenden2016inexact}, although this approach does not effectively exploit second-order derivatives and uses a prescribed space decomposition strategy.

Various other approaches have been proposed for solving problems involving specific regularizers.  In \cite{Liu2009SLEPmanual}, the authors discuss various methods for sparse learning that make use of projection techniques.  A well-known package is {\tt GLMNET} \cite{friedman2010regularization}, which is designed for solving problems with the elastic-net regularization.  Finally, let us mention the work in \cite{Yang2015}, which proposes and tests a groupwise-majorization-descent algorithm (called \gglasso{}) for solving problems involving the group-$\ell_1$-norm regularizer.  A potential downside of this approach is that it updates variables by groups in a cycle, rather than by using an adaptive space decomposition technique. 

\textbf{Second-order methods.}  Relatively few second-order methods have been proposed for minimizing sparsity-promoting objective functions.  In s\cite{grapiglia2019accelerated}, an accelerated regularized Newton scheme is proposed.  A similar proximal-Newton method is proposed in \cite{lee2014proximal}, which under some assumptions can be shown to converge locally superlinearly.  These approaches can be effective in practice, although they appear to lack good worst-case guarantees in terms of identification of the optimal solution support.  Other approaches, such as the orthant-based method in \cite{OBA}, can predict the solution support, but in practice are often outperformed by a closely related method called {\tt FaRSA}~\cite{chen2017reduced,CheCR18}. As for publicly available solvers based on second-order methods, most have been designed for specific loss functions and regularizers. For example, {\tt newGLMNET} in \cite{yuan2012improved} is designed for $\ell_1$-regularized logistic regression and the method in \cite{fan2008liblinear} is  designed for regularized logistic regression and support vector machines.

\subsection{Notation and assumptions}

Let $\R{}$ denote the set of real numbers, $\R{n}$ denote the set of $n$-dimensional real vectors, and $\R{m \times n}$ denote the set of $m$-by-$n$-dimensional real matrices.  The set of natural numbers is denoted as $\N{} := \{0,1,2,\dots\}$.  
For any set $\Ical \subseteq\{1,2,\dots, n\}$, we define the projection of $x\in\R{n}$ onto the subspace spanned by the coordinate vectors indexed by the entries of $\Ical$ as $P_\Ical(x)$, so that
\begin{equation}\label{def:P}
[P_\Ical(x)]_i :=
\begin{cases}
x_i  & \text{if $i\in\Ical$,} \\
0    & \text{if $i\notin\Ical$.}
\end{cases}
\end{equation}
For a function $h: \R{n}\to \R{}$, 
vector $x\in\R{n}$, and direction $d\in\R{n}$, the directional derivative of $h$ at $x$ in the direction $d$ is defined as the following limit:
$$
D_h(x;d): = \lim_{t\searrow 0} \frac{h(x+td) - h(x)}{t}.
$$

The following assumption is assumed to hold throughout the paper.

\bassumption\label{ass.first}
The function $f:\R{n}\to \R{}$ used in the definition of the objective function of problem~\eqref{prob} is convex and continuously differentiable. It follows that there exists a constant $L_f$ such that $\|\nabla f(x)\|_2 \leq L_f$ for all $x\in\Lcal := \{x \in\R{n} : f(x) + r(x) \leq f(x_{0}) + r(x_0)\}$ for any initial estimate $x_0$ of a solution to problem~\eqref{prob}. The objective function $f+r$ is bounded below and the gradient function $\nabla f$ is Lipschitz continuous on $\Lcal$ with Lipschitz constant $L_g$. 
\eassumption

\subsection{Organization}

In Section~\ref{sec.preliminaries}, we present preliminary results related to PG calculations.  In Section~\ref{sec.algorithm}, by using PG-calculations as a starting point, we propose a reduced-space second-order domain decomposition algorithm for solving problem~\eqref{prob}.  The algorithm is analyzed in Section~\ref{sec.analysis} and numerical results are presented in Section~\ref{sec.numerical}.  Finally, in Section~\ref{sec.conclusion}, we provide concluding remarks. 

\section{Preliminaries}\label{sec.preliminaries}

In this section, we discuss preliminary material related to the objective function $f+r$ and its associated PG calculations. (All proofs may be found in Appendix~\ref{app:PG}.) For any $\xbar \in\R{n}$ and $\alphabar>0$, we define the PG \emph{update} as
\bequation\label{def:pg-update}
T(\xbar,\alphabar)
:=\argmin{x\in\R{n}}\left\{\tfrac{1}{2\alphabar}\|x- \big(\xbar-\alphabar \nabla f(\xbar)\big)\|_2^2 + r(x)\right\}
\eequation
and the associated PG \emph{step} as 
\begin{equation}\label{def:pg-step}
s(\xbar,\alphabar)
:= T(\xbar,\alphabar) - \xbar.
\end{equation}
The next result shows that the directional derivative of $f+r$ along the PG step is negative with magnitude proportional to the squared norm of the PG direction. 

\begin{lemma}\label{lem:descent-full}
For any $\xbar \in\R{n}$ and $\alphabar > 0$, the PG step $s(\xbar,\alphabar)$ in~\eqref{def:pg-step} satisfies
$$
D_{f+r}(\xbar;s(\xbar,\alphabar)) 
\leq -\tfrac{1}{\alphabar}\|s(\xbar,\alphabar)\|_2^2.
$$
\end{lemma}

The PG update defined in~\eqref{def:pg-update} can be computed group-wise for each $\Gcal_i\in\Gcal$ by
\begin{equation}\label{def:T}
\begin{aligned}[]
[T(\xbar,\alphabar)]_{\Gcal_i} 
 &= \Big[ \argmin{x\in\R{n}}\Big\{\tfrac{1}{2\alphabar} \| x- \big(\xbar-\alphabar\nabla f(\xbar)\big)\|_2^2 + \sum_{i=1}^{n_\Gcal}\lambda_i\|x_{\Gcal_i}\|_2\Big\}\Big]_{\Gcal_i} \\ 
 &= 
 \max\left\{1 - \frac{\alphabar \lambda_i}{\|[\xbar]_{\Gcal_i}-\alphabar\nabla_{\Gcal_i}f(\xbar)\|_2},0\right\}
  \Big([\xbar]_{\Gcal_i}-\alphabar\nabla_{\Gcal_i}f(\xbar)\Big).
\end{aligned}
\end{equation}
Combining this observation with Lemma~\ref{lem:descent-full} leads to the following corollary, which will be relevant to the manner in which we design the algorithm we propose in Section~\ref{sec.algorithm}.

\begin{lemma}\label{lem:dd} 
For any $\xbar \in\R{n}$, $\alphabar > 0$, and set $\Ical$ equal to the union of a subset of $\{\Gcal_i\}_{i=1}^{n_\Gcal}$, the PG step $s(\xbar,\alphabar)$ defined in~\eqref{def:pg-step} satisfies
\begin{equation}
D_{f+r}\big(\xbar;P_\Ical(s(\xbar,\alphabar))\big) 
\leq -\tfrac{1}{\alphabar}\|P_\Ical(s(\xbar,\alphabar))\|_2^2
\end{equation}
where the projection operator $P_\Ical$ is defined through~\eqref{def:P}.
\end{lemma}

Our next result quantifies the decrease in $f+r$ that one can expect to obtain by taking a PG step $s(\xbar,\alphabar)$, provided the PG parameter $\alphabar$ is sufficiently small.
\begin{lemma}\label{lem:decrease-known}
For any $\xbar \in\R{n}$, $\alphabar\in(0,2/L)$, and $\Ical$ equal to the union of a subset of $\{\Gcal_i\}_{i=1}^{n_\Gcal}$, the objective function decrease satisfies
$$
f(\xbar + P_{\Ical}(\xbar,\sbar)) + r(\xbar + P_{\Ical}(\xbar,\sbar)) 
\leq f(\xbar) + r(\xbar) 
- (\tfrac{1}{\alphabar}-\tfrac{L}{2})\|P_{\Ical}\big(s(\xbar,\alphabar)\big)\|_2^2. 
$$
\end{lemma}

The next result shows that, when restricted to certain groups, the size of the PG step is bounded above by the gradient of the objective function.

\begin{lemma}\label{lem:g>s}
If the pair $(\xbar,\alphabar)$ and group $\Gcal_i$ satisfy $\alphabar\in(0,1]$, $[\xbar]_{\Gcal_i} \neq 0$, and $[\xbar+s(\xbar,\alphabar)]_{\Gcal_i} \neq 0$, where $s(\xbar,\alphabar)$ is defined in~\eqref{def:pg-step}, then
$$
\|\nabla_{\Gcal_i} (f + r)(\xbar)\|_2
\geq \|[s(\xbar,\alphabar)]_{\Gcal_i}\|_2.
$$
\end{lemma}

With the preliminaries now completed, we can propose our new algorithm. 

\section{Proposed Algorithm Framework}\label{sec.algorithm}

We propose Algorithm~\ref{alg:main}, which we call \GroupFaRSA{} ({\tt Fa}st {\tt R}educed-{\tt S}pace {\tt A}lgorithm for {\tt Group} sparsity-inducing regularization) for solving problem~\eqref{prob} that uses ideas related to domain decomposition, subspace acceleration, and support identification.  An overview of the algorithm is described in Section~\ref{sec.alg-main}. During each iteration of our method, at least one of three subroutines is called. The three subroutines are described in Sections~\ref{sec.alg-cgstep}--\ref{sec.alg-lspg}. 

\subsection{Main algorithm (Algorithm~\ref{alg:main})}
\label{sec.alg-main}

Our main algorithm is formally stated as Algorithm~\ref{alg:main}. 
At the beginning of the $k$th iteration, $x_k$ and $\alpha_k > 0$ denote the current  solution estimate for problem~\eqref{prob} and the PG parameter, respectively.
We then compute $\spgk$ in Line~\ref{line:s} as the PG step associated with problem~\eqref{prob}, namely,
\begin{equation}\label{pg-update}
\spgk := s(x_k,\alpha_k) \ \ \text{with} \ \ \text{$s(x_k,\alpha_k)$ defined in~\eqref{def:pg-step}.}
\end{equation}
Although the repeated computation of PG steps is the basis for a first-order method, here we primarily use it to \emph{predict} the zero/nonzero structure of a solution and to formulate optimality measures.  Specifically, in Line~\ref{line:Icgkbar} we compute the index set
\begin{equation}\label{def:Icgkbar}
\begin{aligned}
\Icgkbar &:= \{\text{$j \in \Gcal_i$ : $[x_k]_{\Gcal_i} \neq 0$, \ $[x_k + s_k]_{\Gcal_i} \neq 0$, and} \\
       &\phantom{llllllllllllllllll}
       \text{$\|[x_k]_{\Gcal_i}\|_2 \geq \kappa_1\|\nabla_{\Gcal_i}(f+r)(x_k)\|_2$}\} 
\end{aligned}
\end{equation}
for some $\kappa_1\in (0,\infty)$. The groups of variables that compose $\Icgkbar$ are \emph{candidates} for use in a Newton-type calculation aimed to accelerated convergence.  Before using them, however, we first check to see if each candidate block is sufficiently far from zero, and those that are not are removed.  Specifically, we first define
\begin{equation}\label{def:Iksmall}
\Ismallk := \{j\in\Gcal_i: \Gcal_i\subseteq\Icgkbar \ \text{and} \ \|[x_k]_{\Gcal_i}\|_2 < \kappa_2 \|\nabla_{\Icgkbar}(f+r)(x_k)\|_2^p\}
\end{equation}
for some $\{\kappa_2,p\}\subset (0,\infty)$, and then define in Line~\ref{line:Icgk} the sets and optimality measures 
\begin{equation}\label{def:Icgk}
\left\{
\begin{aligned}
\Icgk &:= \Icgkbar\setminus \Ismallk \\
\Ipgk &:= \{1,2,\dots,n\} \setminus \Icgk
\end{aligned}
\right\}
\quad
\text{and}
\quad
\left\{
\begin{aligned}
\chicgk &:= \|[\spgk]_{\Icgk}\|_2 \\
\chipgk &:= \|[\spgk]_{\Ipgk}\|_2
\end{aligned}
\right\}
\end{equation}
where by convention $\|[\ \cdot\ ]_{\emptyset}\|_2 = 0$. (See Lemma~\ref{optimality-measure} for a justification that these sets together represent a measure of optimality.) This construction of sets also ensures that the subvector of~$x_k$ that corresponds to $\Gcal_i$ for each $\Gcali\subseteq\Icgk$ is at least a distance
\begin{equation}\label{def:rhoki}
\rho_{k,i} 
:= \max\{\kappa_1\|\nabla_{\Gcal_i}(f+r)(x_k)\|_2, \kappa_2\|\nabla_{\Icgk}(f+r)(x_k)\|_2^p\}
\end{equation}
away from zero (see Lemma~\ref{lem:cg-props}(i)), which is crucial in our analysis.

Armed with $\chipgk$ and $\chicgk$, Algorithm~\ref{alg:main} seeks decrease in the objective function in a subspace that is likely to allow for significant progress.  We consider two cases.

\smallskip
\noindent\textbf{Case 1: the condition $\chipgk \leq \chicgk$ checked in Line~\ref{line:main.if} holds.}
In this case, the inequality $\chipgk \leq \chicgk$ indicates that significant reduction in the objective function can be achieved by focusing on variables in the set $\Icgk$. Therefore, in Line~\ref{line:Ik-phi} we choose any index set $\Ical_k$ that is (i) a subset of $\Icgk$, (ii) equal to the union of some subset of groups from $\Gcal$, and (iii) the size of the PG step restricted to the index set $\Ical_k$ is at least a fraction of the size of the PG step when restricted to the index set $\Icgk$.  The easiest choice that satisfies these conditions is $\Ical_k \equiv \Icgk$, but for large-scale problems it may be beneficial to restrict $|\Ical_k|$.  The opposite extreme choice is selecting $\Ical_k$ as the group $\Gcal_i$ contained in $\Icgk$ with largest associated PG step, in which case one would choose $\varphi = 1/\sqrt{n_\Gcal}$ for the user-defined parameter in Line~\ref{line:Ik-phi}. Once $\Ical_k$ has been selected, a \emph{reduced-space} gradient $g_k$ and \emph{reduced-space} positive-definite matrix $H_k$ is computed in Line~\ref{line:gH}, where the derivatives are taken with respect to variables in $\Ical_k$.  (In practice, $H_k$ could be selected based on $\nabla^2_{\Ical_k\Ical_k}(f+r)(x_k)$ to ensure a fast local convergence rate.)  Note that such derivatives exist since by construction $\Ical_k\subseteq\Icgk\subseteq\Icgkbar$, and from~\eqref{def:Icgkbar} the objective function $f+r$ is differentiable with respect to groups of variables in $\Icgkbar$.  Next, $g_k$ and $H_k$ are used to compute a direction $\dbar_k$ of sufficient descent for $f+r$ by calling the subroutine {\sc cg\_direction} (see Section~\ref{sec.alg-cgstep}). Once a full-space vector $d_k$ is obtained by padding $\dbar_k$ with zeros in Line~\ref{line:d-phi}, a  \emph{projected} line search is performed by calling subroutine {\sc update\_cg} in Line~\ref{line:ls-cg} (see Section~\ref{sec.alg-lscg}).

\smallskip
\noindent\textbf{Case 2: the condition $\chipgk \leq \chicgk$ checked in Line~\ref{line:main.if} does not hold.}
In this case, the inequality $\chipgk > \chicgk$ indicates that significant reduction in the objective function can be achieved by focusing on variables in the set $\Ipgk$. Therefore, in Line~\ref{line:Ik-beta}, we choose any index set $\Ical_k$ that is (i) a subset of $\Ipgk$, (ii) equal to the union of some subset of groups from $\Gcal$, and (iii) the size of the PG step restricted to the index set $\Ical_k$ is at least a fraction of the size of the PG step restricted to the index set $\Ipgk$.  The easiest choice that satisfies these conditions is $\Ical_k \equiv \Ipgk$. Once $\Ical_k$ has been chosen, the next iterate is obtained by performing a line search along the PG direction in Line~\ref{line:ls-pg} by calling the subroutine {\sc update\_pg} (for details, see Section~\ref{sec.alg-lspg}). If the subroutine returns $\flagpgk = \decalpha$, the PG parameter is decreased for the next iteration.

\subsection{Computing a CG direction (Algorithm~\ref{alg:cg-step})}
\label{sec.alg-cgstep}

This subroutine returns a reduced-space direction $\dbar_k$ that satisfies conditions~\eqref{line:dbar-1}--\eqref{line:dbar-3}.  We call it a reduced-space vector because the inputs $g_k$ and $H_k$ are elements in $\R{|\Ical_k|}$ and $\R{|\Ical_k|\times|\Ical_k|}$, respectively, where $\Ical_k$ is computed in Line~\ref{line:Ik-phi} of Algorithm~\ref{alg:main}. Condition~\eqref{line:dbar-1} ensures that $\dbar_k$ is a descent direction for the objective function as a consequence of how the reference direction $d_k^R$ is computed in Line~\ref{line:dR}. Condition~\eqref{line:dbar-2} ensures that $\dbar_k$ reduces the model $m_k$ at least as much as a zero step.  Finally, condition~\eqref{line:dbar-3} promotes fast \emph{local} convergence of the iterate sequence $\{x_k\}$ (see Section~\ref{sec.lc}),  but its enforcement (or lack of enforcement) is irrelevant with respect to the complexity result that we prove in Section~\ref{sec.gc}. The subroutine name {\sc cg\_direction} indicates our intent to use the linear CG algorithm in our implementation, although other possible options include a block-wise coordinate descent method applied to the model $m_k$  in~\eqref{def:mk}. In particular, the direction associated with every iteration of the CG algorithm satisfies conditions~\eqref{line:dbar-1}--\eqref{line:dbar-2}, and condition~\eqref{line:dbar-3} is satisfied by all sufficiently large CG iterations. Thus, the requirements of this  subroutine can always be met.

\subsection{Reduced-space search using the CG direction (Algorithm~\ref{alg:ls-cg})}
\label{sec.alg-lscg}

This subroutine performs a search using the direction $d_k$ returned by the subroutine {\sc cg\_direction} in Line~\ref{line:cg-direction} of Algorithm~\ref{alg:main}. For an illustration of this search, which incorporates  projections, see Figure~\ref{fig:projection}. The approach uses the direction $d_k$, without modification, for each block of variables $\Gcal_i$ such that the ray $\{[x_k+\tau d_k]_{\Gcal_i}: \tau \geq 0\}$ does not intersect the ball centered at zero of radius $\rhobar_{k,i} = \min\{\rho_{k,i},\sin(\theta)\|[x_k]_{\Gcal_i}\|_2\}$, where $\rho_{k,i}$ is defined in~\eqref{def:rhoki} and $\theta\in(0,\pi/2)$ is a user-defined parameter. When they do intersect, we first compute $\tau_{k,i}$ as the smallest step along the Newton direction (restricted to block $\Gcal_i$) that intersects the ball.  Then, during the search that follows, anytime the trial step size $\xi^j$ is larger than $\tau_{k,i}$, the trial step for block $\Gcal_i$ is set to zero; otherwise, the Newton direction is used so that the trial step (with respect to block $\Gcal_i$) is $[x_k + \xi^jd_k]_{\Gcal_i}$ (see Line~\ref{line:project}). If termination occurs in Line~\ref{line:ls-phi-nz}, then a new block of variables will become zero, in which case we require the objective function not to increase (see Line~\ref{line:ls-phi-nz-return}).  On the other hand, if termination occurs in Line~\ref{line:ls-phi-return3}, then it indicates that the objective function has been sufficiently reduced (see Line~\ref{line:ls-phi-SD}) and no new groups of zeros have been formed. 

\subsection{Reduced-space line search along a PG step (Algorithm~\ref{alg:ls-beta})}
\label{sec.alg-lspg}

This subroutine performs a line search along the PG direction $P_\Ical(s_k)$.  The search ensures that the next iterate yields decrease in the objective 
of size at least $(\eta\xi^j/\alpha_k)\|P_{\Ical_k}(s_k)\|_2^2$ for some positive integer $j$ computed within the while loop in Line~\ref{line:armijo-pg}.   Once the while loop terminates, the update $\flagpgk \gets \samealpha$ is made if $j = 0$, and set as $\flagpgk \gets \decalpha$ otherwise. The motivation for this update is Lemma~\ref{lem:decrease-known}, which shows that the while loop in Line~\ref{line:armijo-pg} will terminate with $j = 0$ if the PG parameter $\alpha_k$ is sufficiently small.  Therefore, anytime $j > 0$, Algorithm~\ref{alg:ls-beta} returns $\flagpgk \gets \decalpha$ to Algorithm~\ref{alg:main} in Line~\ref{line:ls-pg} so that the PG parameter value for the next iteration is reduced by a factor of $\xi\in(0,1)$ in Line~\ref{line:alpha-decrease}.


\balgorithm[p]
  \caption{\GroupFaRSA{} for solving problem~\eqref{prob}.}
  \label{alg:main}
  \balgorithmic[1]
    \State \textbf{Input:} $x_0$
    \State \textbf{Constants:} $\{\varphi,\xi,\eta,\zeta\}\subset (0,1)$, $\{\kappa_1,\kappa_2,p\} \subset (0,\infty)$, $\theta\in(0,\pi/2)$, and $q \in[1,2]$.
    \State Choose any initial PG parameter $\alpha_0 \in (0,1]$. \label{line:alpha0}
    \For{$k = 0,1,2,\dots$}
      \State Compute the step $\spgk$ from~\eqref{pg-update}.\label{line:s}
      \State Compute the set $\Icgkbar$ 
       from~\eqref{def:Icgkbar}. \label{line:Icgkbar}
      \State Compute $\Icgk$ and $\Ipgk$ and their optimality measures $\chicgk$ and $\chipgk$ from~\eqref{def:Icgk}. \label{line:Icgk}
      \If{$\chipgk \leq \chicgk$} \label{line:main.if} 
        \State \label{line:Ik-phi} Choose any $\Ical_k\subseteq\Icgk$ such that 
        $$
        \|[\spgk]_{\Ical_k}\|_2 \geq \varphi\|[\spgk]_{\Icgk}\|_2 \equiv \varphi\chicgk \text{ and } \Ical_k \text{ is the union of some } \{\Gcal_j\}.
        $$  
        \State Set $g_k \gets \nabla_{\Ical_k} (f + r)(x_k)$ and pick a positive-definite $H_k\in\R{|\Ical_k|\times|\Ical_k|}$ . \label{line:gH}
        \State Call Algorithm~\ref{alg:cg-step} to obtain $\dbar_k \gets$ {\sc cg\_direction}$(g_k,H_k)$. \label{line:cg-direction}
        \State Set $[d_k]_{\Ical_k} \gets \dbar_k$ and $[d_k]_{\Ical_k^c} \gets 0$. \label{line:d-phi}
        \State Call Algorithm~\ref{alg:ls-cg} to obtain $(x_{k+1},\flagcgk) \gets$ {\sc update\_cg}$(x_k,d_k,\Ical_k)$.\label{line:ls-cg}
        \State Set $\alpha_{k+1} \gets \alpha_k$.
      \Else 
        \State \label{line:Ik-beta} Choose any $\Ical_k\subseteq\Ipgk$ such that  
        $$
        \|[\spgk]_{\Ical_k}\|_2 \geq \varphi \|[\spgk]_{\Ipgk}\|_2 \equiv \varphi \chipgk \text{ and } \Ical_k \text{ is the union of some } \{\Gcal_j\}.
        $$
        \State Call Algorithm~\ref{alg:ls-beta} to obtain $(x_{k+1},\flagpgk) \gets$ {\sc update\_pg}$(x_k,\spgk,\alpha_k,\Ical_k)$. \label{line:ls-pg}
      \If{$\flagpgk = \decalpha$} \label{line:main-alpha-if}
         \State $\alpha_{k+1} \gets \zeta\alpha_k$.\label{line:alpha-decrease}
      \Else   
         \State $\alpha_{k+1} \gets \alpha_k$.\label{line:alpha-same}
      \EndIf
      \EndIf   
    \EndFor
  \algstore{bkbreak}
  \ealgorithmic
\ealgorithm

\balgorithm[p]
  \caption{Computing $\dbar_k$ in Line~\ref{line:cg-direction} of Algorithm~\ref{alg:main}.}
  \label{alg:cg-step}
  \balgorithmic[1]
  \algrestore{bkbreak}
  \Procedure{$\dbar_k =$ cg\_direction}{$g_k,H_k$}
  \State \textbf{Constant:} $q$ is provided by Algorithm~\ref{alg:main}.
  \State Define the model
        \bequation\label{def:mk}
           m_k(d) := g_k^T d + \thalf d^T H_k d.
        \eequation
  \State Compute the reference direction (an approximate minimizer of $m_k$) as \label{line:dR}
        $$
          d_k^{R} \gets -\beta_k g_k, \ \text{where} \ \beta_k \gets \|g_k\|_2^2/(g_k^T H_k g_k).
        $$
        \State Choose $\mu_k\in(0,1]$ and then compute any $\dbar_k \approx \argmin{d} \ m_k(d)$ that satisfies \label{line:dbar}
        \begin{align}
          g_k^T \dbar_k &\leq g_k^T d_k^R,  \label{line:dbar-1} \\ 
          m_k(\dbar_k) &\leq m_k(0), \ \ \text{and}  \label{line:dbar-2} \\ 
          \|H_k\dbar_k+g_k\|_2 &\leq \mu_k \|g_k\|_2^q. \label{line:dbar-3}
        \end{align}
    \State \textbf{return} $\dbar_k$ 
  \EndProcedure
  \algstore{bkbreak}
  \ealgorithmic
\ealgorithm

\balgorithm[t]
  \caption{Computing $x_{k+1}$ in Line~\ref{line:ls-cg} of Algorithm~\ref{alg:main}.}
  \label{alg:ls-cg}
  \balgorithmic[1]
  \algrestore{bkbreak}
  \Procedure{$(x_{k+1},\flagcgk) =$ update\_cg}{$x_k,d_k,\Ical_k$}
    \State \textbf{Constants:} $\eta$, $\xi$, and $\theta$ provided by Algorithm~\ref{alg:main}.
    \For{each $i$ such that $\Gcal_i\subseteq\Ical_k$}
      \State Compute $\rho_{k,i}$ as defined in~\eqref{def:rhoki}.
      \State Set $\rhobar_{k,i} \gets \min\{\rho_{k,i},\sin(\theta)\|[x_k]_{\Gcal_i}\|_2\}$. \label{line:rhobarki}
      \If{$\{[x_k+\tau d_k]_{\Gcal_i}: \tau \geq 0\} \cap \{x\in\R{|\Gcal_i|}: \|x\|_2 \leq \rhobar_{k,i}\} = \emptyset$}
         \State Set $\tau_{k,i} \gets \infty$.
      \Else
         \State Set $\tau_{k,i}$ as the smallest positive root of $\|[x_k+\tau d_k]_{\Gcal_i}\|_2 = \rhobar_{k,i}$.
      \EndIf   
    \EndFor
    \State Set $j \gets 0$ and $\tau_k :=  \min_{i}\{\tau_{k,i} : \Gcal_i\subseteq\Ical_k\}$. \label{line:def-tauk}
    \While{$\xi^j \geq \tau_k$} \label{line:update-cg-while}
       \State Set $[y_j]_{\Ical_k^c} \gets [x_k]_{\Ical_k^c}$.
       \For{each $i$ such that $\Gcal_i\in\Ical_k$}
       \State Set $
       [y_j]_{\Gcal_i} \gets
       \begin{cases}
       [x_k]_{\Gcal_i} + \xi^j [d_k]_{\Gcal_i} & \text{if $\xi^j < \tau_{k,i}$,} \\
       0 & \text{if $\xi^j \geq \tau_{k,i}$.} 
       \end{cases}
       \label{line:project}
       $
       \EndFor
       \If{$f(y_j) + r(y_j) \leq f(x_k) + r(x_k)$} \label{line:ls-phi-nz}
         \State \textbf{return} $x_{k+1} \gets y_j$ and $\flagcgk \gets \zb$ \label{line:ls-phi-nz-return} 
       \EndIf
       \State Set $j \gets j + 1$.
    \EndWhile
    \Loop\label{line:update-cg-loop}
       \State Set $y_j \gets x_k+\xi^jd_k$.
     \If{$f(y_j) + r(y_j) \leq f(x_k) + r(x_k) + \eta \xi^j \nabla_{\Ical_k}(f + r)(x_k)^T [d_k]_{\Ical_k}$} \label{line:ls-phi-SD}
        \State \textbf{return} $x_{k+1} \gets y_j$ and $\flagcgk \gets \sd$
         \label{line:ls-phi-return3} 
      \EndIf
      \State Set $j \gets j + 1$. \label{line:j-inc} 
    \EndLoop
  \EndProcedure
  \algstore{bkbreak}
  \ealgorithmic
\ealgorithm

\begin{figure}[h!]
\centering
\includegraphics[scale=0.4,trim=0 0 0 0,clip]{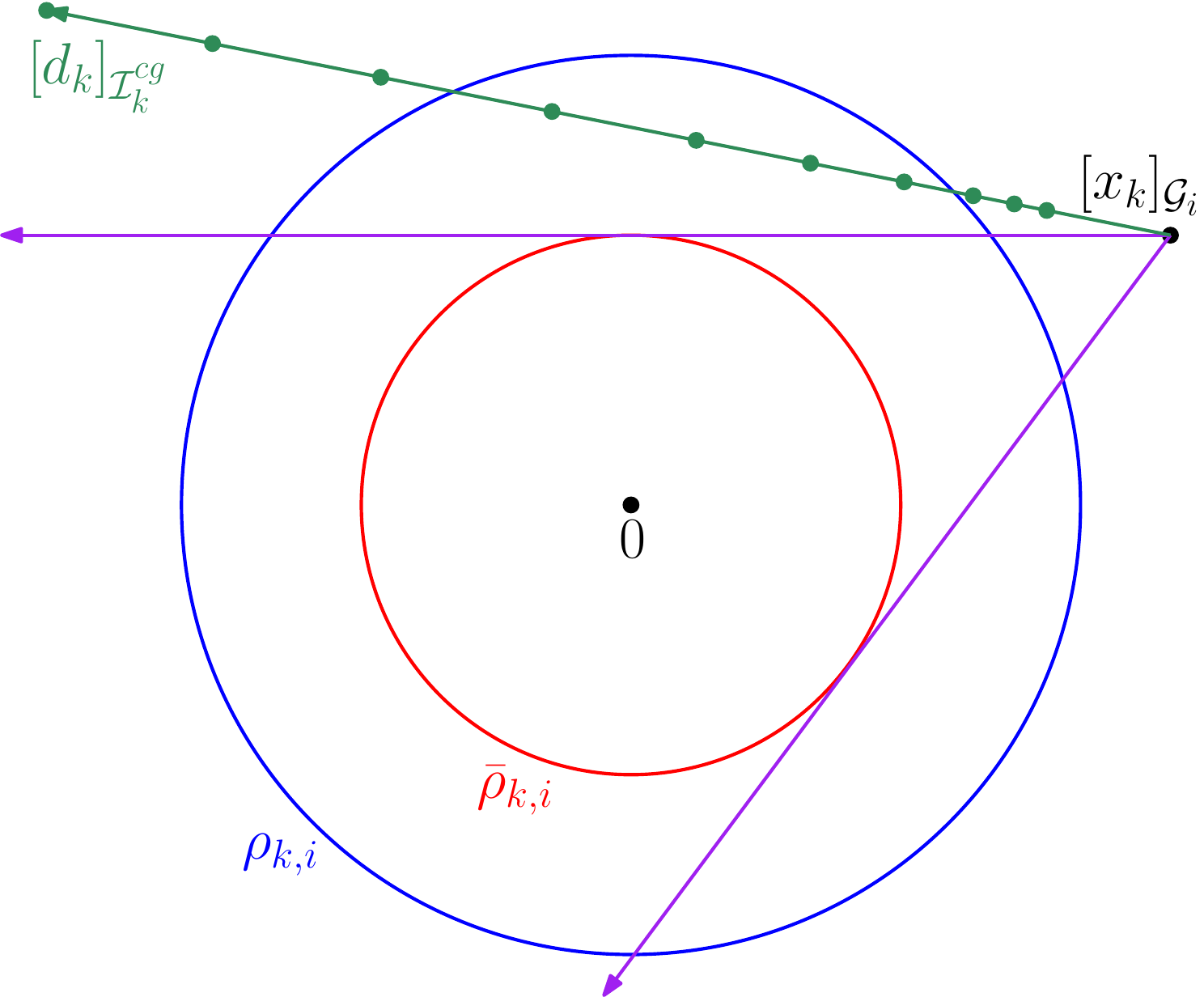}
\includegraphics[scale=0.4,trim=0 0 0 0,clip]{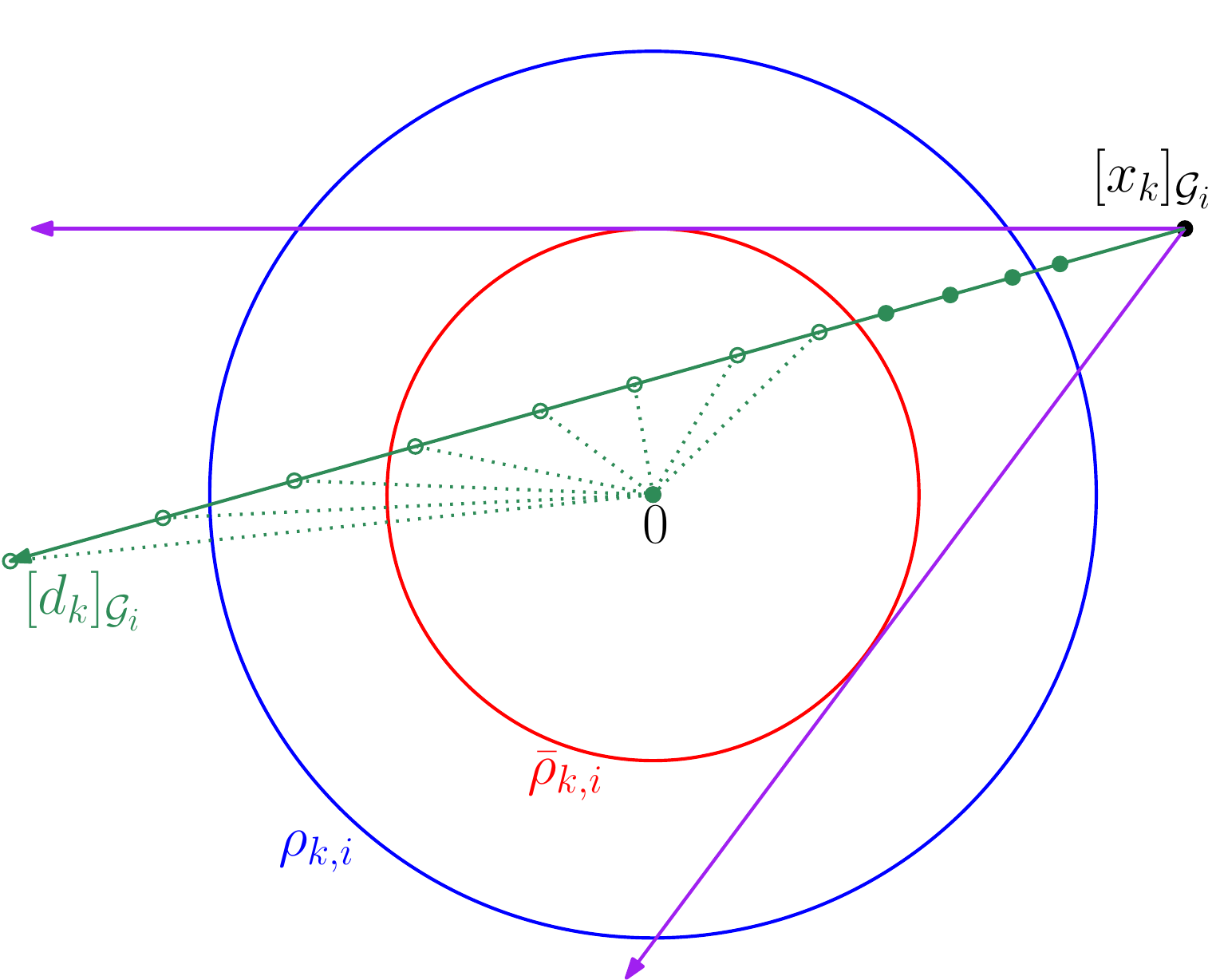}
\caption{\label{fig:projection} The reduced-space projected search based on the Newton-CG direction $d_k$ described in Section~\ref{sec.alg-lscg}. In the figure on the left, the direction $d_k$ does not intersect the ball of radius $\rhobarki$.  In this case, standard backtracking is used, as indicated by the solid green dots. In the figure on the right, the direction $d_k$ does intersect the ball of radius $\rhobarki$.  In this case, all points after the first point of intersection (indicated by hollow green circles) are projected to zero.  Once the backtracking points leave the ball of radius $\rhobarki$ (indicated as solid green dots), standard backtracking is resumed.}
\end{figure}

\balgorithm[t]
  \caption{Computing $x_{k+1}$ in Line~\ref{line:ls-pg} of Algorithm~\ref{alg:main}.}
  \label{alg:ls-beta}
  \balgorithmic[1]
  \algrestore{bkbreak}
  \Procedure{$(x_{k+1}, \flagpgk) =$ update\_pg}{$x_k,\spgk,\alpha_k,\Ical_k$}
    \State \textbf{Constants:} $\eta$ and $\xi$ provided by Algorithm~\ref{alg:main}.
    \State Set $j \gets 0$ and $y_0 \gets x_k + P_{\Ical_k}(\spgk)$.\label{line:y0-pg}
    \While{$f(y_j) + r(y_j) > f(x_k) + r(x_k) - \eta\xi^j \tfrac{1}{\alpha_k}\|P_{\Ical_k}(\spgk)\|_2^2$}\label{line:armijo-pg}
      \State Set $j \gets j + 1$ and then $y_j \gets x_k + \xi^j P_{\Ical_k}(\spgk)$. \label{line:y-up-pg}
    \EndWhile
    \If{$j = 0$} \label{line:update-pg-if}
       \State \textbf{return} $x_{k+1} \gets y_{j}$ and $\flagpgk \gets \samealpha$ \label{line:update2}
    \Else
       \State \textbf{return} $x_{k+1} \gets y_{j}$ and $\flagpgk \gets \decalpha$ \label{line:update2b}
    \EndIf   
  \EndProcedure
  \ealgorithmic
\ealgorithm
 

\section{Analysis}\label{sec.analysis}

Our analysis considers worst-case complexity (Section~\ref{sec.gc}) and local convergence (Section~\ref{sec.lc}) properties  of Algorithm~\ref{alg:main}.  
To identify an approximate solution to problem~\eqref{prob}, we use the measure $\max\{\chipgk,\chicgk\}$, as we now justify.

\begin{lemma}\label{optimality-measure} 
Let $\Kcal\subseteq\N{}$ be such that $\lim_{k\in\Kcal} x_k = x_*$ and $\lim_{k\in\Kcal} \alpha_k = \alpha_* > 0$. Then, $x_*$ is a solution to problem~\eqref{prob} if and only if $\lim_{k\in\Kcal} \max\{\chipgk,\chicgk\} = 0$.
\end{lemma}
\begin{proof}
First, we may apply~\cite[Theorem~3.2.8]{ConGT00a}, with the choice $y = (\xbar,\alphabar)$ and the set map $\Ccal(y) = \R{n}$, to the objective function appearing in~\eqref{def:pg-update} to conclude that $T(\xbar,\alphabar)$ is continuous on $\R{n} \times (0,\infty)$.  Combining this property with the definition of $T$ in~\eqref{def:pg-update} and the assumption that $\lim_{k\in\Kcal}(x_k,\alpha_k) = (x_*,\alpha_*)$ with $\alpha_* > 0$ shows that
$\lim_{k\in\Kcal} s_k
= \lim_{k\in\Kcal} \big( T(x_k,\alpha_k) - x_k \big) 
=  T(x_*,\alpha_*) - x_*.$
It follows from this limit and the fact that  Assumption~\ref{ass.first} and~\cite[Theorem~10.7]{beck2017first} together show that $x_*$ is a solution to problem~\eqref{prob} if and only if $T(x_*,\alpha_*) = x_*$.
\end{proof}

If $\max\{\chicgk,\chipgk\} = 0$ for some $k \in \N{}$, then Lemma~\ref{optimality-measure} implies that $x_k$ is a solution to problem~\eqref{prob}.  Hence, all that remains is to consider the behavior of Algorithm~\ref{alg:main} when an infinite number of iterations is performed.  To focus on this case, we make the following assumption, which is assumed to hold throughout the rest of this section.

\begin{assumption}\label{ass.nonzero}
For all iterations $k\in\N{}$, it holds that $\max\{\chicgk,\chipgk\} > 0$.
\end{assumption}

Since our analysis considers the properties of the sequence of iterates, it is convenient to define the following partition of iterations performed by Algorithm~\ref{alg:main}:
\begin{align*}
\setcg   &:= \{k\in\N{}: \text{Line~\ref{line:ls-cg} is reached during the $k$th iteration}\}, \\
\setcgZERO &:= \{k\in\setcg: \text{subroutine {\sc update\_cg} returns $\flagcgk = \zb$ in Line~\ref{line:ls-cg}}\}, \\
\setcgSD &:= \{k\in\setcg: \text{subroutine {\sc update\_cg} returns $\flagcgk\! =\! \sd$ in Line~\ref{line:ls-cg}}\}, \\
\setpg   &:= \{k\in\N{}: \text{Line~\ref{line:ls-pg} is reached during the $k$th iteration}\}, \\
\setpgSAME   &:= \{k\in\setpg: \text{subroutine {\sc update\_pg} returns $\flagpgk = \samealpha$ in Line~\ref{line:ls-pg}}\}, \ \text{and} \\
\setpgDEC   &:= \{k\in\setpg: \text{subroutine {\sc update\_pg} returns $\flagpgk = \decalpha$ in Line~\ref{line:ls-pg}}\},
\end{align*}
so that
$\setcg = \setcgZERO \cup \setcgSD$,
$\setpg = \setpgSAME \cup \setpgDEC$, and
$\N{} = \setcg\cup\setpg$.

Finally, we assume that the symmetric and positive-definite matrices required in Line~\ref{line:gH} are chosen to be bounded and uniformly positive definite.

\begin{assumption}\label{ass:Hk}
The matrix sequence $\{H_k\}_{k\in\setcg}$ chosen in Line~\ref{line:gH} is bounded and uniformly positive definite.  That is, there exist constants $0 < \mumin \leq \mumax < \infty$ such that $\mumin \|v\|_2^2  \leq v^T H_k v \leq \mumax \|v\|_2^2$ for all $k\in\setcg$ and $v\in\R{|\Ik|}$.
\end{assumption}

\subsection{Complexity result}
\label{sec.gc}

We first focus our attention on iterations in $\setpg$.  The next result shows  that Algorithm~\ref{alg:ls-beta} is well posed and that the new iterate that it produces satisfies a decrease property that will be useful for our complexity analysis.

\begin{lemma}\label{lem:K-pg}
For each $k\in \setpg$, Algorithm~\ref{alg:ls-beta} is called in Line~\ref{line:ls-pg} and successfully returns $x_{k+1}$ and $\flagpgk$.  Moreover, the value of $\flagpgk$ indicates whether $k\in\setpgDEC$ or $k\in\setpgSAME$, and for these respective cases the following properties hold:
\begin{itemize}
\item[(i)] If $k\in\setpgSAME$, then $\alpha_{k+1} = \alpha_k$ and
\begin{equation}\label{result:dec-pg}
f(x_{k+1}) + r(x_{k+1})
\leq f(x_k) + r(x_k) - \tfrac{\eta\varphi^2}{\alpha_k} (\chipgk)^2.
\end{equation}
\item[(ii)] If $k\in\setpgDEC$, then $\alpha_{k+1} = \xi\alpha_k$ and
$
f(x_{k+1}) + r(x_{k+1})
< f(x_k) + r(x_k).
$
\end{itemize}
\end{lemma}
\begin{proof}
Since $k\in\setpg$, we know that the condition tested in Line~\ref{line:main.if} of Algorithm~\ref{alg:main} must not hold, meaning that $\chipgk > \chicgk$. Combining this observation with Line~\ref{line:Ik-beta} of Algorithm~\ref{alg:main} shows that the set $\Ical_k$ defined in Line~\ref{line:Ik-beta} satisfies
\begin{equation}\label{eq:P-bd}
\|P_{\Ical_k}(s_k)\|_2
= \|[s_k]_{\Ical_k}\|_2
\geq \varphi \chipgk > 0.
\end{equation}
Combining this result with Lemma~\ref{lem:dd} (using $\Ical = \Ik$, $\xbar = x_k$, and $\alphabar = \alpha_k$) yields
\begin{equation}\label{eq:sd-pg}
D_{f+r}(x_k; P_{\Ical_k}(s_k)) \leq -\tfrac{1}{\alpha_k}\|P_{\Ical_k}(s_k)\|_2^2 < 0.
\end{equation}

It is possible that Algorithm~\ref{alg:ls-beta} terminates in Line~\ref{line:update2} because the inequality in Line~\ref{line:armijo-pg} does not hold for $j = 0$. In this case, Algorithm~\ref{alg:ls-beta} successfully returns $x_{k+1} = y_0 = x_k + P_{\Ical_k}(s_k)$ and $\flagpgk = \samealpha$, also indicating that $k\in\setpgSAME$.  Since the while-loop in Line~\ref{line:armijo-pg} terminates with $j = 0$, we can conclude that
\begin{equation}\label{dec-pg-1}
f(x_{k+1}) + r(x_{k+1})
\equiv f(y_0) + r(y_0)
\leq f(x_k) + r(x_k) - \tfrac{\eta}{\alpha_k}\|P_{\Ical_k}(\spgk)\|_2^2.
\end{equation}
Combining this inequality with~\eqref{eq:P-bd} shows that \eqref{result:dec-pg} holds. Finally, since  $\flagpgk = \samealpha$, it follows from Line~\ref{line:alpha-same} that $\alpha_{k+1} = \alpha_k$, completing the proof in this case.

It remains to consider the case when Algorithm~\ref{alg:ls-beta} is unable to terminate in Line~\ref{line:update2} because the inequality in Line~\ref{line:armijo-pg} holds for $j = 0$.  In this case, it follows from \eqref{eq:sd-pg} and standard results for a backtracking Armijo line search that, for all sufficiently large $j$, the vector $y_j \gets x_k + \xi^j P_{\Ical_k}(s_k)$ defined in Line~\ref{line:y-up-pg} of Algorithm~\ref{alg:ls-beta} satisfies
\begin{equation}\label{dec-nsuff}
\begin{aligned}
f(y_j) + r(y_j) 
&\leq f(x_k) + r(x_k) + \eta\xi^j D_{f+r}(x_k; P_{\Ical_k}(s_k)) \\
&\leq f(x_k) + r(x_k) - \eta\xi^j \tfrac{1}{\alpha_k}\|P_{\Ical_k}(\spgk)\|_2^2.
\end{aligned}
\end{equation}
This inequality shows that the while loop starting in Line~\ref{line:armijo-pg} of Algorithm~\ref{alg:ls-beta} will terminate finitely, and thus Algorithm~\ref{alg:ls-beta} successfully returns $x_{k+1} = y_j = x_k + \xi^j P_{\Ical_k}(s_k)$ for some $j > 0$ and $\flagpgk = \decalpha$, also indicating that $k\in\setpgDEC$. Combining~\eqref{dec-nsuff}, $y_j = x_{k+1}$, and \eqref{eq:sd-pg} proves that $f(x_{k+1}) + r(x_{k+1}) < f(x_k) + r(x_k)$, as claimed. Finally, since $\flagpgk = \decalpha$, we see in Line~\ref{line:alpha-decrease} that $\alpha_{k+1} = \xi\alpha_k$.
\end{proof}

Next, we prove that the PG parameter remains bounded away from zero.

\begin{lemma} \label{lem:alpha-fixed}
The PG parameter sequence generated by Algorithm~\ref{alg:main} satisfies 
\begin{equation}\label{def:alphamin}
1 \geq \alpha_k 
\geq \alpha_{\min} 
:= \min\big\{\alpha_0,\tfrac{2\xi(1-\eta)}{L}\big\} > 0 
\ \ \text{for all $k\in\N{}$.}
\end{equation}
Moreover, a bound on the number of times the PG parameter is decreased is given by
\begin{equation}\label{bound-Kalphadecr}
|\setpgDEC|
\leq 
c^\alpha_{\downarrow} 
:= \max\left\{0, \left\lceil \log\Big(\tfrac{\alpha_0 L}{2(1-\eta)}\Big)
/ \log(\xi^{-1})\right\rceil \right\}. 
\end{equation}
\end{lemma}
\begin{proof}
We first prove~\eqref{def:alphamin}. Since $\alpha_0 \in (0,1]$ in Line~\ref{line:alpha0} and $\alpha_{k+1} \leq \alpha_k$ for all $k\in\N{}$, we need only prove the lower bound on $\alpha_k$ in~\eqref{def:alphamin}. With that goal in mind, for the purpose of obtaining a contradiction, suppose that there exists an iteration $k$ satisfying $\alpha_k \leq 2(1-\eta)/L < 2/L$, with the latter inequality holding since $\eta\in(0,1)$. 

First suppose that $k\in\setpg$.  With $y_0 = x_k + P_{\Ical_k}(s_k)$ as defined in Line~\ref{line:y0-pg} of Algorithm~\ref{alg:ls-beta}, it follows from Lemma~\ref{lem:decrease-known} with $\xbar = x_k$, $\alphabar = \alpha_k$, and $s(\xbar,\alphabar) = s_k$ that
\begin{align*}
f(y_0) + r(y_0) 
&\leq f(x_k) + r(x_k) - (\tfrac{1}{\alpha_k}-\tfrac{L}{2})\|P_{\Ical}(s_k)\|_2^2 \\
&\leq f(x_k) + r(x_k) - \big(\tfrac{1}{\alpha_k}-\tfrac{2(1-\eta)}{2\alpha_k}\big) \|P_{\Ical}(s_k)\|_2^2 \\
&= f(x_k) + r(x_k) - \tfrac{\eta}{\alpha_k}\|P_{\Ical}(s_k)\|_2^2.
\end{align*}
This inequality implies that the condition checked in Line~\ref{line:armijo-pg} for $j = 0$ will not hold, meaning that $j = 0$ when Line~\ref{line:update-pg-if} is reached so that $\flagpgk \gets \samealpha$ in Line~\ref{line:update2}. Thus, when Line~\ref{line:main-alpha-if} in Algorithm~\ref{alg:main} is reached, the update $\alpha_{k+1} \gets \alpha_k$ will take place.  Second, if $k\in\setcg$, then Algorithm~\ref{alg:main} sets $\alpha_{k+1} \gets \alpha_k$. To summarize, anytime $\alpha_k \leq 2(1-\eta)/L$, the update $\alpha_{k+1} \gets \alpha_k$ takes place. Combining this property with the fact that when the PG parameter is decreased the update $\alpha_{k+1} \gets \xi\alpha_k$ is used (see Line~\ref{line:alpha-decrease} in Algorithm~\ref{alg:main}), shows that~\eqref{def:alphamin} holds.

We now prove~\eqref{bound-Kalphadecr}. Let us observe from the first paragraph in  this proof that if $\alpha_0 \leq 2(1-\eta)/L$ then $|\setpgDEC| = 0$, which verifies that~\eqref{bound-Kalphadecr} holds. Therefore, for the remainder of the proof, suppose that $\alpha_0 > 2(1-\eta)/L$. Combining this bound with the fact that when the PG parameter is decreased the update $\alpha_{k+1} \gets \xi\alpha_k$ is used, we can see that an upper bound on $|\setpgDEC|$ is the smallest integer $\ell$ such that $\alpha_0\xi^\ell \leq 2(1-\eta)/L$.  Solving this inequality for $\ell$ shows that the result in~\eqref{bound-Kalphadecr} holds.
\end{proof}

We now switch our attention to iterations in $\setcg$. The next result establishes that Algorithm~\ref{alg:cg-step} is well posed, and that the direction $d_k$ that results from it when called by Algorithm~\ref{alg:main} satisfies a certain descent property.

\begin{lemma}\label{lem:is-descent}
For each $k\in\setcg$, Algorithm~\ref{alg:cg-step} is well posed.  Moreover, the  resulting direction $\dbar_k$, which is used to compute $d_k$ in Line~\ref{line:d-phi}, guarantees that $d_k$ satisfies
\begin{itemize}
\item[(i)] $\nabla_{\Ical_k}(f + r)(x_k)^T [d_k]_{\Ical_k}
\leq -\tfrac{1}{\mumax}\|\nabla_{\Ical_k}(f + r)(x_k)\|_2^2
< 0$, and
\item[(ii)] $\|d_k\|_2 \leq (2/\mumin)\|\nabla_\Ik (f + r)(x_k) \|_2$ 
\end{itemize}
where $\Ical_k\subseteq\Icgk$ is the set in Line~\ref{line:Ik-phi} used as an input to Algorithm~\ref{alg:cg-step} in Line~\ref{line:ls-cg}.
\end{lemma}
\begin{proof}
Since $k\in\setcg$, Algorithm~\ref{alg:cg-step} is called in Line~\ref{line:cg-direction} with input $\Ik$ defined in Line~\ref{line:Ik-phi}.  We first prove that $g_k = \nabla_{\Ical_k}(f + r)(x_k)$, as defined in Line~\ref{line:gH}, is nonzero.  For a proof by contradiction, suppose that $g_k = 0$ so that $\nabla_{\Gcali}(f + r)(x_k) = 0$ for all $i$ such that $\Gcali\subseteq\Ik$.  Consider arbitrary such $i$.  Note that $[x_k]_\Gcali \neq 0$ and $[x_k+s_k]_\Gcali \neq 0$ since $\Gcali\subseteq\Ik\subseteq\Icgk$ (see Line~\ref{line:Ik-phi}) and by how $\Icgk$ is defined. This allows us to conclude from Lemma~\ref{lem:g>s} that $[s_k]_\Gcali = 0$, i.e., that $[s_k]_\Ik = 0$ since $i$ with $\Gcali \subseteq \Ik$ was arbitrary.  This fact and Line~\ref{line:Ik-phi} yields $\chicgk = 0$, but since the inequality in Line~\ref{line:main.if} must hold, we also have $\chipgk = 0$.  This contradicts Assumption~\ref{ass.nonzero}, thus establishing that $g_k \neq 0$. Now, it follows from Lines~\ref{line:gH}, \ref{line:d-phi}, \ref{line:dbar}, and \ref{line:dR}, $g_k \neq 0$, and Assumption~\ref{ass:Hk} that
\begin{equation*} 
\begin{aligned}
\nabla_{\Ical_k}(f + r)(x_k)^T [d_k]_{\Ical_k}
&\equiv g_k^T\dbar_k 
\leq g_k^T d_k^R 
= -\beta_k\|g_k\|_2^2  \\
&= -\|g_k\|_2^4 / (g_k^T H_k g_k) 
\leq -\tfrac{1}{\mumax}\|g_k\|_2^2. 
\end{aligned}
\end{equation*}
The result in (i) follows from this inequality and $g_k = \nabla_{\Ical_k}(f + r)(x_k) \neq 0$.

Part (ii) is precisely~\cite[Lemma~3.8]{chen2017reduced} under our Assumption~\ref{ass:Hk} since our conditions placed upon the step $d_k$ are exactly the same as those used in~\cite{chen2017reduced}.
\end{proof}

The next lemma shows that, for $k\in\setcg$, a local Lipschitz property holds along a certain portion of the search path defined by the reduced-space Newton-CG direction.

\begin{lemma}\label{lem:cg-props}
Let $k\in\setcg$ so that $\Ical_k$ is computed in Line~\ref{line:Ik-phi}.
The following hold:
\begin{itemize}
\item[(i)] The constant $\theta\in(0,\pi/2)$ and index set $\Ik$ passed into Algorithm~\ref{alg:ls-cg} satisfy, 
for each $i$ such that $\Gcal_i\subseteq\Ical_k$ with $\rho_{k,i}$ computed in~\eqref{def:rhoki} and $\rhobar_{k,i}$ computed in Line~\ref{line:rhobarki}, the following conditions:
\begin{itemize}
\item[(a)] $\|[x_k+s_k]_\Gcali\|_2 \neq 0$, 
\item[(b)] $\|[x_k]_{\Gcal_i}\|_2 
\geq \rho_{k,i}
\geq \rhobar_{k,i} 
\geq \sin(\theta) \rho_{k,i} > 0$, and 
\item[(c)] $\|[x_k]_\Gcali\|_2 - \rhobar_{k,i} \geq \kappa_2(1-\sin(\theta))\|\nabla_\Ik (f+r)(x_k)\|_2^p$.
\end{itemize}
\item[(ii)] For all step sizes $\beta\in[0,\tau_k)$ with $\tau_k$ computed in Line~\ref{line:def-tauk}, it holds, with 
\begin{equation}\label{def:rhokmin}
\lammax := \max\{\lambda_1,\lambda_2,\dots, \lambda_{n_\Gcal}\}
\ \ \text{and} \ \
\rhokmin := \min_{i}\{\rho_{k,i}: \Gcal_i\subseteq\Ik\}
\end{equation}
that
$\|\nabla_\Ik (f+r)(x_k) - \nabla_\Ik (f+r)(x_k+\beta d_k)\|_2 \leq \beta\big(L + \tfrac{\lammax}{\rhokmin}\big)\|[d_k]_\Ik\|_2$.
\end{itemize}
\end{lemma}
\begin{proof}
We first prove part (i). Consider arbitrary $i$ with $\Gcal_i \subseteq\Ical_k$, where $\Ical_k\subseteq\Icgk$ is passed into Algorithm~\ref{alg:ls-cg} and constructed to satisfy the condition in Line~\ref{line:Ik-phi}. Part (a) follows from $\Icgk\subseteq\Icgkbar$ and the definition of $\Icgkbar$ in~\eqref{def:Icgkbar}. The first inequality in part (b) follows from $\Icgk\subseteq\Icgkbar$, and how $\Icgk$, $\Ismallk$, and $\Icgkbar$ are defined. The second inequality in (b) follows from how $\rhobar_{k,i}$ is defined in Line~\ref{line:rhobarki}. The third inequality in (b) follows from Line~\ref{line:rhobarki} and the first inequality in (b). To complete the proof for part (b), we must prove that $\rho_{k,i} > 0$. For a proof by contradiction, assume that $\rho_{k,i} = 0$, which by~\eqref{def:rhoki} means that $\|\nabla_{\Icgk} (f+r)(x_k)\|_2 = 0$.  It follows from this fact that each $i$ with $\Gcali \subseteq\Ik\subseteq\Icgk$  satisfies $\|\nabla_{\Gcali} (f+r)(x_k)\|_2 = 0$, which in light of Lemma~\ref{lem:g>s} (using $\xbar = x_k$, $\alphabar = \alpha_k$, and $s(\xbar,\alphabar) = s_k$) and the definition of $\Icgk$ implies that $\|[s_k]_\Gcali\|_2 = 0$ for each $\Gcali \subseteq\Ik$, i.e., that $\|[s_k]_\Ik\|_2 = 0$. It now follows from Line~\ref{line:Ik-phi} that $\chicgk = 0$, which combined with the inequality in Line~\ref{line:main.if} shows that $\chipgk$ = 0.  Since we have reached a contradiction to Assumption~\ref{ass.nonzero}, we must conclude that $\rho_{k,i} > 0$, as claimed. Finally, we aim to prove part (c).  It follows from
Line~\ref{line:rhobarki}, $\theta \in(0,\pi/2)$,  part (b), \eqref{def:rhoki}, and the fact that $\Ik \subseteq \Icgk$ that 
\begin{align*}
\|[x_k]_\Gcali\|_2 - \rhobar_{k,i}
&\geq \|[x_k]_\Gcali\|_2 - \sin(\theta)\|[x_k]_\Gcali\|_2
=(1-\sin(\theta))\|[x_k]_\Gcali\|_2 \\
&\geq (1-\sin(\theta))\rhoki
\geq \kappa_2(1-\sin(\theta))\|\nabla_{\Icgk}(f+r)(x_k)\|_2^p \\
&\geq \kappa_2(1-\sin(\theta))\|\nabla_{\Ik}(f+r)(x_k)\|_2^p,
\end{align*}
which completes the proof of part (c).

To prove part (ii), let $\beta\in[0,\tau_k)$.   It follows from part (i) 
and the definition of $\tau_k$ in Line~\ref{line:def-tauk} that every point on the segment that connects $[x_k]_\Gcali$ to $[x_k+\beta d_k]_\Gcali$ is outside of  the ball in $\R{|\Gcali|}$ centered at zero of radius $\rhobar_{k,i} > 0$.  This means that both $\|[x_k]_{\Gcali}\| \geq \rhobar_{k,i}$ and $\|[x_k+\beta d_k]_{\Gcali}\| \geq \rhobar_{k,i}$. It now follows that 
\begin{equation}\label{r-grad-diff}
\begin{aligned}
&\ \|\nabla_\Gcali r(x_k) - \nabla_\Gcali r(x_k+\beta d_k)\|_2 \\
&= \lambda_i \left\|\frac{[x_k]_\Gcali}{\|[x_k]_\Gcali\|_2} - \frac{[x_k+\beta d_k]_\Gcali}{\|[x_k+\beta d_k]_\Gcali\|_2}\right\|_2 
= \frac{\lambda_i}{\rhobar_{k,i}} \left\|\frac{\rhobar_{k,i}[x_k]_\Gcali}{\|[x_k]_\Gcali\|_2} - \frac{\rhobar_{k,i}[x_k+\beta d_k]_\Gcali}{\|[x_k+\beta d_k]_\Gcali\|_2}\right\|_2 \\
&\leq \frac{\lambda_i}{\rhobar_{k,i}} \|[x_k]_\Gcali - [x_k+\beta d_k]_\Gcali\|_2 
= \frac{\lambda_i\beta}{\rhobar_{k,i}} \|[d_k]_\Gcali\|_2,
\end{aligned}
\end{equation}
where the (only) inequality follows from the nonexpansive property of the projection (of $[x_k]_\Gcali$ and $[x_k+ \beta d_k]_\Gcali$) onto the ball of radius $\rhobar_{k,i}$.  From \eqref{r-grad-diff} we have
\begin{align}
&\ \|\nabla_\Ik r(x_k) - \nabla_\Ik r(x_k+\beta d_k) \|_2^2 \nonumber \\
&= \sum_{i:\Gcali\subseteq\Ik} \|\nabla_\Gcali r(x_k) - \nabla_\Gcali r(x_k+\beta d_k)\|_2^2  
\leq  \beta^2\sum_{i:\Gcali\subseteq\Ik}  \frac{\lambda_i^2}{\rhobar_{k,i}^2} \|[d_k]_\Gcali\|_2^2 \nonumber \\
&\leq  \frac{\beta^2 \lammax^2}{\rhokmin^2}\sum_{i:\Gcali\subseteq\Ik} \|[d_k]_\Gcali\|_2^2 
=  \frac{\beta^2 \lammax^2}{\rhokmin^2} \|[d_k]_\Ik\|_2^2.
 \label{bd:Lip-r}
\end{align}
It follows from Assumption~\ref{ass.first}, $[d_k]_{\Ical_k^c} = 0$, the triangle inequality, and~\eqref{bd:Lip-r} that 
\begin{align*}
&\ \|\nabla_\Ik (f+r)(x_k) - \nabla_\Ik (f+r)(x_k+\beta d_k)\|_2  \\
&\leq \| \nabla_\Ik f(x_k) - \nabla_\Ik f(x_k+\beta d_k)\|_2 + \|\nabla_\Ik r(x_k) - \nabla_\Ik r(x_k+\beta d_k) \|_2 \\
&\leq L\beta\|d_k\|_2 + \(\beta\tfrac{\lammax}{\rhokmin}\)\|[d_k]_\Ik\|_2 
=  \beta\(L + \tfrac{\lammax}{\rhokmin}\)\|[d_k]_\Ik\|_2,
\end{align*}
which completes the proof.
\end{proof}

We now show that Algorithm~\ref{alg:ls-beta} is well posed and that the new iterate it  produces satisfies a decrease property that will be used in the final complexity result.

\begin{lemma}\label{lem:dec-cg-wp}
For each $k\in\setcg$, Algorithm~\ref{alg:ls-cg} is called in Line~\ref{line:ls-cg} and successfully returns $x_{k+1}$ and $\flagcgk$.  Moreover, the value of $\flagcgk$ indicates whether $k\in \setcgZERO$ or $k\in\setcgSD$, and for these respective cases the following properties hold:
\begin{itemize}
\item[(i)] If $k\in\setcgZERO$, then  $f(x_{k+1})+r(x_{k+1}) \leq f(x_k) + r(x_k)$ and $x_{k+1}$ has at least one additional block of zeros compared to $x_k$.
\item[(ii)] If $k\in\setcgSD$, then 
\begin{equation}\label{result:dec-cg}
f(x_{k+1}) + r(x_{k+1})
\leq f(x_k) + r(x_k) - \min\{c_1(\chicgk)^{1+p},c_2(\chicgk)^{2+p}\}
\end{equation}
where 
\begin{equation}\label{def:cs}
\begin{aligned}
c_1 &:= 
\frac{\eta\xi\mumin\kappa_2\big(1-\sin(\theta)\big)\varphi^{1+p}}{2\mumax} > 0
\ \ \text{and} \\
c_2 &:= \frac{\kappa_2\mumin^2 \xi\eta(1-\eta)\varphi^{2+p}}{2\mumax^2\big(L \kappa_2 (L_f+\lammax\sqrt{n_\Gcal})^p + \lammax\big)} > 0.
\end{aligned}
\end{equation}
\end{itemize}
\end{lemma}
\begin{proof}
Throughout, we use $F := f + r$.  It is possible that Algorithm~\ref{alg:ls-cg} successfully terminates in Line~\ref{line:ls-phi-nz-return}, in which case it follows from Line~\ref{line:ls-phi-nz-return} and Line~\ref{line:ls-phi-nz} that the returned $x_{k+1}$ and $\flagcgk$ satisfy $F(x_{k+1}) \leq F(x_k)$ and $\flagcgk = \zb$, indicating that $k\in\setcgZERO$. Moreover, upon termination, the value $j$ satisfies $\xi^j \geq \tau_k$ (see Line~\ref{line:update-cg-while}), which combined with Line~\ref{line:project} shows that at least one additional group of variables has become zero at $x_{k+1}$.  This proves that part (i) holds.

Next, suppose that Algorithm~\ref{alg:ls-cg} does not terminate in Line~\ref{line:ls-phi-nz-return}.  Observe from the definition of $\tau_k$ in Line~\ref{line:def-tauk} that $\tau_k > 0$ (this follows from Lemma~\ref{lem:cg-props}(i) and the definition of $\rhobar_{k,i}$). Therefore, it follows that the \textbf{while} loop starting in Line~\ref{line:update-cg-while} will terminate with the smallest nonnegative integer $\jbar$ such that $\xi^{\jbar} < \tau_k$, and the \textbf{loop} in Line~\ref{line:update-cg-loop} will begin with $j = \jbar$.
We now claim that the condition in Line~\ref{line:ls-phi-SD} used to determine termination of the \textbf{loop} is satisfied for all $j \geq \jbar$ such that
\begin{equation}\label{xij-interval}
\xi^j \in \left[0, \frac{2(\eta-1)\nabla_\Ik F(x_k)^T [d_k]_\Ik}{(L+\lammax/\rhokmin)\|[d_k]_\Ik\|_2^2}\right] \subset [0,\tau_k).
\end{equation}
To see that this claim holds, we can use the integral form of Taylor's Theorem and Lemma~\ref{lem:cg-props}(ii) (using the fact that $\gamma\xi^j\in [0,\tau_k)$ for all $\gamma\in[0,1]$) to obtain
\begin{align*}
|F(x_k &+ \xi^j d_k) - F(x_k) - \xi^j \nabla_\Ik F(x_k)^T [d_k]_\Ik | \\\
&\leq \left| \int_{0}^1 \xi^j [d_k]_\Ik^T \big(\nabla_\Ik F(x_k+\gamma\xi^j d_k) - \nabla_\Ik F(x_k) \big) \text{d}\gamma \right| \\
&\leq \xi^j \int_{0}^1 \|[d_k]_\Ik\|_2 \|\nabla_\Ik F(x_k+\gamma\xi^j d_k) - \nabla_\Ik F(x_k) \big)\|_2 \text{d}\gamma \\
&\leq \xi^{2j} (L+\lammax/\rhokmin)\|[d_k]_\Ik\|_2^2 \int_{0}^1 \gamma  \text{d}\gamma 
= \thalf \xi^{2j} (L+\lammax/\rhokmin) \|[d_k]_\Ik\|_2^2.  
\end{align*}
Combining this inequality with~\eqref{xij-interval} yields
\begin{align*}
F(x_k+\xi^j d_k) 
&\leq F(x_k) + \xi^j \nabla_\Ik F(x_k)^T [d_k]_\Ik +  \thalf \xi^{2j} (L+\lammax/\rhokmin) \|[d_k]_\Ik\|_2^2 \\
&= F(x_k) + \xi^j \nabla_\Ik F(x_k)^T [d_k]_\Ik + \xi^j(\eta-1)\nabla_\Ik F(x_k)^T [d_k]_\Ik \\
&= F(x_k) + \eta\xi^j \nabla_\Ik F(x_k)^T [d_k]_\Ik,
\end{align*}
which establishes our claim that the inequality in Line~\ref{line:ls-phi-SD} holds for all $j \geq \jbar$ such that $\xi^j$ satisfies~\eqref{xij-interval}. This shows that the \textbf{loop} will successfully terminate with $\flagcgk = \sd$ (thus indicating that $k\in\setcgSD$) and  $x_{k+1}$ satisfying
\begin{equation}\label{dec-gen}
F(x_{k+1}) 
\leq F(x_k) + \eta\xi^{\jhat} \nabla_\Ik F(x_k)^T [d_k]_\Ik
\end{equation}
for some $\jhat$ satisfying
\begin{align}
\xi^{\jhat} 
&\geq \min\left\{\xi^{\jbar}, \frac{2\xi(\eta-1)\nabla_\Ik F(x_k)^T [d_k]_\Ik}{(L+\lammax/\rhokmin)\|[d_k]_\Ik\|_2^2}\right\} \nonumber \\
&\geq \min\left\{\xi\tau_k, \frac{2\xi(\eta-1)\nabla_\Ik F(x_k)^T [d_k]_\Ik}{(L+\lammax/\rhokmin)\|[d_k]_\Ik\|_2^2}\right\} \label{3-cases}
\end{align}
where the second inequality follows from the fact that $\jbar$ is the \emph{smallest} nonnegative integer such that $\xi^{\jbar} < \tau_k$. We now consider two cases.

\noindent\textbf{Case 1:} the minimum in~\eqref{3-cases} is $\xi\tau_k$, from which we may conclude that $\tau_k < \infty$. Using~\eqref{dec-gen} and Lemma~\ref{lem:is-descent}(i) we have that
\begin{equation}\label{dec-case-2}
\begin{aligned}
F(x_{k+1}) 
&\leq F(x_k) + \eta\xi^{\jhat} \nabla_\Ik F(x_k)^T [d_k]_\Ik 
\leq F(x_k) - \tfrac{\eta\xi}{\mumax}\tau_k\|\nabla_{\Ical_k} F(x_k) \|_2^2.
\end{aligned}
\end{equation}
We now seek a lower bound on $\tau_k$. Consider $i$ such that $\tau_{k,i} < \infty$ when computed in Algorithm~\ref{alg:ls-cg}. The triangle inequality gives 
$\rhobar_{k,i} = \|[x_k+\tau_{k,i}d_k]_\Gcali\|_2
\geq \|[x_k]_\Gcali\|_2 - \tau_{k,i}\|[d_k]_\Gcali\|_2$, 
which together with Lemma~\ref{lem:cg-props}(i)(c) and Lemma~\ref{lem:is-descent}(ii) shows that
\begin{align*}
\tau_{k,i} 
&\geq \frac{\|[x_k]_\Gcali\|_2 - \rhobar_{k,i}}{\|[d_k]_\Gcali\|_2} \\
&\geq \frac{\mumin\kappa_2 (1-\sin(\theta))\|\nabla_\Ik F(x_k)\|_2^p }{2\|\nabla_\Ik F(x_k) \|_2}
= \thalf\mumin\kappa_2 (1-\sin(\theta)) \|\nabla_\Ik F(x_k)\|_2^{p-1}.
\end{align*}
From this, it follows that $\tau_k \geq \thalf\mumin\kappa_2 (1-\sin(\theta))\|\nabla_\Ik F(x_k)\|_2^{p-1}$.  Using this inequality  with~\eqref{dec-case-2},  Lemma~\ref{lem:g>s}, 
and the set $\Ik$ from  Line~\ref{line:Ik-phi} shows that 
\begin{align*}
F(x_{k+1}) 
&\leq F(x_k) - \frac{\eta\xi\mumin\kappa_2(1-\sin(\theta))}{2\mumax}  \|\nabla_{\Ical_k} F(x_k) \|_2^{1+p} \\
&\leq F(x_k) - \frac{\eta\xi\mumin\kappa_2(1-\sin(\theta))}{2\mumax}  \|[s_k]_\Ik \|_2^{1+p} \\
&\leq F(x_k) - \frac{\eta\xi\mumin\kappa_2\big(1-\sin(\theta)\big)\varphi^{1+p}}{2\mumax}(\chicgk)^{1+p},
\end{align*}
thus completing the proof for this case.

\noindent\textbf{Case 2:} the minimum in~\eqref{3-cases} is $\frac{2\xi(\eta-1)\nabla_\Ik F(x_k)^T [d_k]_\Ik}{(L+\lammax/\rhokmin)\|[d_k]_\Ik\|_2^2}$. Combining this fact with \eqref{dec-gen}, \eqref{3-cases}, Lemma~\ref{lem:is-descent}(i), and Lemma~\ref{lem:is-descent}(ii) shows that
\begin{equation}\label{dec-3}
\begin{aligned}
F(x_{k+1}) 
&\leq F(x_k) + \eta\xi^{\jhat} \nabla_\Ik F(x_k)^T [d_k]_\Ik \\
&\leq F(x_k) - \frac{2\xi\eta(1-\eta)\|\nabla_\Ik F(x_k)\|_2^4}{\mumax^2(L+\lammax/\rhokmin)\|[d_k]_\Ik\|_2^2} \\
&\leq F(x_k) - \frac{2\mumin^2 \xi\eta(1-\eta)\|\nabla_\Ik F(x_k)\|_2^4}{4\mumax^2(L+\lammax/\rhokmin)\|\nabla_\Ik F(x_k)\|_2^2} \\
&= F(x_k) - \frac{\mumin^2 \xi\eta(1-\eta)\|\nabla_\Ik F(x_k)\|_2^2}{2\mumax^2(L+\lammax/\rhokmin)}.
\end{aligned}
\end{equation} 
It follows from~\eqref{def:rhokmin}, \eqref{def:rhoki}, and $\Ik\subseteq\Icgk$ that $\rhokmin \geq \kappa_2\|\nabla_\Ik F(x_k)\|_2^p$.  Combining this bound with~\eqref{dec-3} shows that
\begin{equation}\label{dec-3b}
\begin{aligned}
F(x_{k+1}) 
&\leq F(x_k) - \frac{\mumin^2 \xi\eta(1-\eta)\|\nabla_\Ik F(x_k)\|_2^2}{2\mumax^2(L+\lammax/\rhokmin)} \\
&\leq F(x_k) - \frac{\mumin^2 \xi\eta(1-\eta)\|\nabla_\Ik F(x_k)\|_2^2}{2\mumax^2\big(L+\lammax/(\kappa_2\|\nabla_\Ik F(x_k)\|_2^p)\big)} \\
&= F(x_k) - \frac{\kappa_2\mumin^2 \xi\eta(1-\eta)\|\nabla_\Ik F(x_k)\|_2^{2+p}}{2\mumax^2(L \kappa_2\|\nabla_\Ik F(x_k)\|_2^p + \lammax)}.
\end{aligned}
\end{equation}
Next, we know from Lemma~\ref{lem:K-pg}, Lemma~\ref{lem:dec-cg-wp}(i), and
equations~\eqref{dec-case-2} and \eqref{dec-3b} that $F(x_k) \leq F(x_0)$ for all $k\in\N{}$, i.e., $x_k\in\Lcal$ for all $k\in\N{}$.  Combining this fact with the triangle inequality, Assumption~\ref{ass.first}, the definition of $r$, and~\eqref{def:rhokmin} gives
\begin{align*}
\|\nabla_\Ik F(x_k)\|_2
&\leq \|\nabla_\Ik f(x_k)\|_2 + \|\nabla_\Ik r(x_k)\|_2 \\
&= \|\nabla_\Ik f(x_k)\|_2 + \Big( \sum_{i:\Gcali\subseteq\Ik}\|\nabla_\Gcali r(x_k)\|_2^2\Big)^{1/2}  \\
&\leq L_f + \Big( \sum_{i:\Gcali\subseteq\Ik}\|\lambda_i [x_k]_\Gcali / \|[x_k]_\Gcali\|_2\|_2^2\Big)^{1/2} \\
&=  L_f + \Big( \sum_{i:\Gcali\subseteq\Ik} \lambda_i^2 \Big)^{1/2} 
\leq L_f +  \Big( \sum_{i:\Gcali\subseteq\Ik} \lammax^2 \Big)^{1/2} 
\leq L_f +  \lammax \sqrt{n_\Gcal}.
\end{align*}
Combining this with~\eqref{dec-3b} gives
$$
F(x_{k+1}) 
\leq F(x_k) - \left(\frac{\kappa_2\mumin^2 \xi\eta(1-\eta)}{2\mumax^2\big(L \kappa_2 (L_f+\lammax\sqrt{n_\Gcal})^p + \lammax\big)}\right) \|\nabla_\Ik F(x_k)\|_2^{2+p}, 
$$
which combined with Lemma~\ref{lem:g>s} 
and how the index set $\Ik$ in Line~\ref{line:Ik-phi} is defined gives
\begin{align*}
F(x_{k+1}) 
&\leq F(x_k) - \left(\frac{\kappa_2\mumin^2 \xi\eta(1-\eta)}{2\mumax^2\big(L \kappa_2 (L_f+\lammax\sqrt{n_\Gcal})^p + \lammax\big)}\right) \|[s_k]_\Ik\|_2^{2+p} \\
&\leq F(x_k) - \left(\frac{\kappa_2\mumin^2 \xi\eta(1-\eta)\varphi^{2+p}}{2\mumax^2\big(L \kappa_2 (L_f+\lammax\sqrt{n_\Gcal})^p + \lammax\big)}\right) (\chicgk)^{2+p},
\end{align*}
thus completing the proof.
\end{proof}

The result in~\eqref{result:dec-cg} motivates us to define the following subsets of $\setcgSD$:
\begin{equation}\label{eq.frank}
\setcgSDbig := \{k\in\setcgSD: \chicgk \geq c_1/c_2\} \ \ \text{and} \ \
\setcgSDsmall := \setcgSD\setminus\setcgSDbig. 
\end{equation}
This distinction plays a role in our complexity result.  First, we 
require a lemma.

\begin{lemma}\label{obj-monotonic}
The objective function $f+r$ is monotonically decreasing over the sequence of iterates $\{x_k\}$ and $\lim_{k\to\infty} \big(f(x_k) + r(x_k) \big) 
=: \lb > -\infty$.
\end{lemma}
\begin{proof}
It follows from Lemma~\ref{lem:K-pg} and Lemma~\ref{lem:dec-cg-wp} that the objective function is monotonically decreasing over the iterate sequence. The remaining conclusion of the lemma follows from the monotonicity property and Assumption~\ref{ass.first}.   
\end{proof}

The main theorem can now be stated.  It gives an upper bound on the number of iterations performed by Algorithm~\ref{alg:main} before an approximate solution is obtained.

\begin{theorem}\label{thm:complexity}
Let $c_1$ and $c_2$ be the constants defined in~\eqref{def:cs} and let us define $c_3 := \eta\varphi^2/\alpha_0 > 0$. For any $\epsilon > 0$, define $\setepsilon := \{k \in \N{} : \max\{\chicgk,\chipgk\} > \epsilon\}$.  Then, 
\begin{equation}\label{final-set-bds}
\begin{aligned}
|\setpgSAME \cap \setepsilon|
&\leq c_{\mathrm{pg}}\epsilon^{-2} + 1, \\
|\setcgSDbig \cap \setepsilon| 
&\leq  c_{\mathrm{big}}\epsilon^{-(1+p)} + 1, \ \ \text{and} \\
|\setcgSDsmall \cap \setepsilon| 
&\leq c_{\mathrm{small}}\epsilon^{-(2+p)} + 1
\end{aligned} 
\end{equation} 
where the constants $c_{\mathrm{pg}}$, $c_{\mathrm{big}}$, and $c_{\mathrm{small}}$ are given, respectively, by
\begin{equation} \label{def:cbig}
\begin{aligned}
c_{\mathrm{pg}}  &:= \big( f(x_0) + r(x_0) - \lb \big) / c_3, \\
c_{\mathrm{big}} &:= \big( f(x_0) + r(x_0) - \lb \big) / c_1, \ \ \text{and} \\
c_{\mathrm{small}} &:= \big( f(x_0) + r(x_0) - \lb \big) / c_2.
\end{aligned}
\end{equation}
Therefore, if $\epsilon \geq c_1/c_2$, then
\begin{equation} \label{result:big}
|\setepsilon| \leq \big( c^\alpha_\downarrow + c_{\mathrm{pg}}\epsilon^{-2} + c_{\mathrm{big}}\epsilon^{-(1+p)} + 2\big)(1+ n_\Gcal) + n_\Gcal
\end{equation}
where $c^\alpha_\downarrow$ is defined in~\eqref{bound-Kalphadecr}; otherwise, i.e., if $\epsilon < c_1/c_2$, then
\begin{equation} \label{result:small}
|\setepsilon| \leq \big( c^\alpha_\downarrow + c_{\mathrm{pg}}\epsilon^{-2} + c_{\mathrm{big}}\epsilon^{-(1+p)} + c_{\mathrm{small}}\epsilon^{-(2+p)}  + 3\big)(1+ n_\Gcal) + n_\Gcal.
\end{equation}
\end{theorem}
\begin{proof}
Note that the definitions of $\setcg$ and $\setpg$ together with Line~\ref{line:main.if} show that
\begin{equation}\label{chi-inequalities}
\chicgk \geq \chipgk \ \text{for $k\in \setcg$} 
\ \ \text{and} \ \ 
\chipgk > \chicgk  \ \text{for $k\in\setpg$.} 
\end{equation}
Define $\Delta_k := f(x_k) + r(x_k) - \big( f(x_{k+1}) + r(x_{k+1})\big)$ and $m_k := \max\{\chipgk,\chicgk\}$.  Using Lemma~\ref{lem:K-pg}(i), Lemma~\ref{lem:alpha-fixed}, Lemma~\ref{lem:dec-cg-wp}(ii), the definitions of $c_3$ and $\setepsilon$ in the statement of the theorem, and \eqref{chi-inequalities} shows for arbitrary $\kbar \in \N{}$ that
\begin{equation*}
\begin{aligned}
&f(x_0) + r(x_0) - \big( f(x_{\kbar+1})+r(x_{\kbar+1}) \big) = \sum_{0 \leq k \leq \kbar} \Delta_k \\
&\geq \sum_{\substack{k\in\setpgSAME \cap \setepsilon \\ 0 \leq k \leq \kbar}} \Delta_k
    + \sum_{\substack{k\in\setcgSDbig \cap \setepsilon \\ 0 \leq k \leq \kbar}} \Delta_k
    + \sum_{\substack{k\in\setcgSDsmall \cap \setepsilon \\ 0 \leq k \leq \kbar}} \Delta_k
    \\
&\geq 
 \sum_{\substack{k\in\setpgSAME \cap \setepsilon \\ 0 \leq k \leq \kbar}} c_3(\chipgk)^2 
     + \sum_{\substack{k\in\setcgSDbig \cap \setepsilon \\ 0 \leq k \leq \kbar}} c_1(\chicgk)^{1+p} 
     + \sum_{\substack{k\in\setcgSDsmall \cap \setepsilon \\ 0 \leq k \leq \kbar}} c_2(\chicgk)^{2+p} \\
&=  \sum_{\substack{k\in\setpgSAME \cap \setepsilon \\ 0 \leq k \leq \kbar}} c_3 m_k^2
     + \sum_{\substack{k\in\setcgSDbig \cap \setepsilon \\ 0 \leq k \leq \kbar}} c_1 m_k^{1+p} 
     + \sum_{\substack{k\in\setcgSDsmall \cap \setepsilon \\ 0 \leq k \leq \kbar}} c_2 m_k^{2+p} \\
&\geq \sum_{\substack{k\in\setpgSAME \cap \setepsilon \\ 0 \leq k \leq \kbar}} c_3 \epsilon^2 
     + \sum_{\substack{k\in\setcgSDbig \cap \setepsilon \\ 0 \leq k \leq \kbar}} c_1 \epsilon^{1+p}
     + \sum_{\substack{k\in\setcgSDsmall \cap \setepsilon \\ 0 \leq k \leq \kbar}} c_2 \epsilon^{2+p}. 
\end{aligned}
\end{equation*}
From this inequality, Lemma~\ref{obj-monotonic}, and~\eqref{def:cbig} one finds that \eqref{final-set-bds} follows.

Next, suppose that $\epsilon \geq c_1/c_2$. 
It then follows from~\eqref{eq.frank} and \eqref{chi-inequalities} that 
$\chicgk
= \max\{\chipgk,\chicgk\}
> \epsilon \geq c_1/c_2$
for all $k\in\setcg$, 
which implies that $\setcgSDsmall \cap \setepsilon = \emptyset$.  The result in~\eqref{result:big} follows from this observation, \eqref{final-set-bds}, \eqref{bound-Kalphadecr}, and since (by Lemma~\ref{lem:dec-cg-wp}(i)) at most $n_\Gcal$ iterations in $\setcgZERO$ can occur before the first, after the last, or between any two iterations in $\setpgDEC\cup\setpgSAME\cup\setcgSD$. 

The final result~\eqref{result:small} follows using the same argument as in the previous paragraph, except now $\setcgSDsmall \cap \setepsilon$ is no longer necessarily empty.
\end{proof}

We see from~\eqref{result:small} that, for all sufficiently small $\epsilon$, the worst case complexity result for Algorithm~\ref{alg:main} is $\epsilon^{-(2+p)}$, which is worse than the $\epsilon^{-2}$ result that holds for the PG method.  If one is concerned with such a result, the difference can be made arbitrarily small (for a range of $\epsilon$ values typically used in practice) by choosing $p$ sufficiently small. However, as is typical with well-designed second-derivative methods, although the complexity bound is worse, it typically performs better (see Section~\ref{sec.numerical}).

\subsection{Local convergence}\label{sec.lc}

We now consider the local convergence rate of the iterates generated by Algorithm~\ref{alg:main}.  Our analysis is performed under the following additional assumption that will be assumed to hold throughout this section.

\begin{assumption}\label{ass.lc}
The function $f$ is twice continuously differentiable and strongly convex.
It follows that  that there exists a unique solution $x_*$ to the optimization problem~\eqref{prob} with optimal support $\Scal_* := \{i: [x_*]_\Gcali \neq 0\}$. 
Moreover, we assume that $\nabla^2 f: \R{n}\to\R{n\times n}$ is Lipschitz continuous in a neighborhood of the solution $x_*$, and that  $f+r$ is nondegenerate at $x_*$ in the sense that $\|[\nabla f(x_*)]_\Gcali\|_2 < \lambda_i$ for all $i\notin\Scal_*$. 
\end{assumption}

Optimality conditions for problem~\eqref{prob} imply that $\|[\nabla f(x_*)]_\Gcali \|_2 \leq \lambda_i$ for all $i\notin\Scal_*$.  Thus, the final condition in Assumption~\ref{ass.lc} is a strengthening of this fact.

\begin{assumption}\label{ass.Ik-all}
The following algorithmic choices are made in Algorithm~\ref{alg:main}:
\begin{itemize}
\item[(i)] The backtracking parameter is chosen to satisfy $\eta\in(0,1/2)$. 
\item[(ii)] For all sufficiently large $k \in \N{}$, $\Ik$ in Line~\ref{line:Ik-phi}/\ref{line:Ik-beta} is chosen as
\begin{equation}\label{Ik-is-all}
\Ik =
\begin{cases}
\Icgk & \text{if $k\in\setcg$,} \\ 
\Ipgk & \text{if $k\in\setpg$.}
\end{cases}
\end{equation}
\item[(iii)] For all sufficiently large $k\in\setcg$, $H_k = \nabla^2_{\Ik\Ik} (f+r)(x_k)$ is chosen in Line~\ref{line:gH}.
\end{itemize}
\end{assumption}

The next result establishes that the iterate sequence converges to $x_*$.

\ifreport
\begin{theorem}\label{thm:its-converge}
The iterate sequence $\{x_k\}$ generated by Algorithm~\ref{alg:main} satisfies 
$$
\lim_{k\to\infty} x_k = x_*
\ \ \text{and}  \ \ 
\lim_{k\to\infty} \max\{\chipgk,\chicgk\} = 0.
$$
\end{theorem}
\else
  \begin{theorem}\label{thm:its-converge}
  The iterate sequence $\{x_k\}$ generated by Algorithm~\ref{alg:main} satisfies $\lim_{k\to\infty} x_k = x_*$ and $\lim_{k\to\infty} \max\{\chipgk,\chicgk\} = 0$.
  \end{theorem}
\fi
\begin{proof}
Theorem~\ref{thm:complexity} gives $\lim_{k\to\infty} \max\{\chipgk,\chicgk\} = 0$.  Since $\{x_k\}$ is bounded due to monotonicity of $\{f(x_k)+r(x_k)\}$ (see Lemma~\ref{obj-monotonic}) and Assumption~\ref{ass.lc}, there exists an infinite $\Kcal \subseteq \N{}$ and $\xhat$ so that $\lim_{k\in\Kcal,k\to\infty} x_k = \xhat$.  It follows from Lemma~\ref{optimality-measure} and Lemma~\ref{lem:alpha-fixed} that $\xhat$ is a solution to problem~\eqref{prob}, but with Assumption~\ref{ass.lc} this means that $\xhat = x_*$, so $\lim_{k\in\Kcal,k\to\infty} x_k = x_*$. The fact that the \emph{entire} sequence $\{x_k\}$ converges to $x_*$ follows from this fact, Assumption~\ref{ass.lc}, and monotonicity of $\{f(x_k)+r(x_k)\}$.
\end{proof}

We now show for groups whose variables are all equal to zero at the solution that the PG step will eventually predict them to be zero.

\begin{lemma}\label{lem:not-in-support}
For all $i\notin\Scal_*$ and sufficiently large $k$, it holds that $[x_k+s_k]_\Gcali = 0$.
\end{lemma}
\begin{proof}
First note that Lemma~\ref{lem:alpha-fixed} and the update strategy for $\{\alpha_k\}$ in Algorithm~\ref{alg:main} ensure that there exists $\kbar_1$ such that $\alpha_k = \alpha_* > 0$ for all $k \geq \kbar_1$. Now, let $i\notin\Scal_*$ so that $[x_*]_\Gcali = 0$. It follows from  Assumption~\ref{ass.lc} that
\begin{equation*}
\frac{\alpha_* \lambda_i}{\|[x_*-\alpha_*\nabla f(x_*)]_\Gcali\|_2}
= \frac{\lambda_i}{\|[\nabla f(x_*)]_\Gcali \|_2} > 1.
\end{equation*}
Combining this with Theorem~\ref{thm:its-converge}, $\alpha_k = \alpha_* > 0$ for all $k \geq \kbar_1$, and Assumption~\ref{ass.first} shows that there exists a $\kbar_2 \geq \kbar_1$ such that $1-\alpha_k\lambda_i/\|[x_k - \alpha_k\nabla f(x_k)]_\Gcali\|_2 < 0$ for all $k\geq \kbar_2$. Using this fact with~\eqref{def:pg-step} and \eqref{def:T} shows that $[x_k+s_k]_\Gcali = 0$ for all $k\geq \kbar_2$.  This completes the proof since the choice $i\notin\Scal_*$ was arbitrary and $n_\Gcal$ is finite.
\end{proof}

We now show that, eventually, the set $\Scal_*$ determines the sets $\Ipgk$ and $\Icgk$. 

\begin{lemma}\label{lem:Ipg-Icg-correct}
For all sufficiently large $k$, it holds that 
$$
\Ipgk \equiv \{j\in\Gcali: i\notin\Scal_*\}
\ \ \text{and} \ \
\Icgk \equiv \{j\in\Gcali: i\in\Scal_*\}
$$
where the sets $\Ipgk$ and $\Icgk$ are defined in~\eqref{def:Icgk}.
\end{lemma}
\begin{proof}
Let $\kbar_1$ be large enough so that the conclusion of Lemma~\ref{lem:not-in-support} holds, i.e., if $k\geq \kbar_1$ and $i\notin \Scal_*$, then $[x_k+s_k]_\Gcali = 0$.   Together with~\eqref{def:Icgkbar}, this shows that $\Gcali\cap\Icgkbar = \emptyset$ for all $k \geq \kbar_1$ and $i\notin \Scal_*$, and thus $\Gcali\subseteq\Ipgk$ (see~\eqref{def:Icgk}) for all $k\geq \kbar_1$ and $i\notin\Scal_*$.  In other words, it holds that $\{j\in\Gcali:i\notin\Scal_*\}\subseteq \Ipgk$ for all $k\geq \kbar_1$. 

Next, we prove that there exists $\kbar_2$ such that $\Ipgk \subseteq \{j\in\Gcali:i\notin\Scal_*\}$ for all $k\geq \kbar_2$. For a proof by contradiction, suppose that there exists an infinite subsequence $\Kcal\subseteq\N{}$ and group index $\ibar$ such that $\Gcal_{\ibar} \subseteq\Ipgk$ and $\ibar \in\Scal_*$ for all $k\in\Kcal$.  
Since $\Gcal_{\ibar}\subseteq\Ipgk$ for all $k\in\Kcal$, it follows from~\eqref{def:Icgkbar}, \eqref{def:Iksmall}, and~\eqref{def:Icgk} that at least one of
\begin{align}
[x_k]_{\Gcal_{\ibar}} = 0, \ \
[x_k+s_k]_{\Gcal_{\ibar}} = 0, \ \ 
\|[x_k]_{\Gcal_{\ibar}}\|_2 &< \kappa_1\|\nabla_{\Gcal_{\ibar}}(f+r)(x_k)\|_2\, \ \ 
\text{or} \label{conds-hold-1} \\ 
\|[x_k]_{\Gcal_{\ibar}}\|_2 &< \kappa_2 \|\nabla_{\Icgkbar}(f+r)(x_k)\|_2^p \label{conds-hold-2}
\end{align}
holds for all $k\in\Kcal$. However, since $\ibar\in\Scal_*$, it follows from Theorem~\ref{thm:its-converge} that the first condition in~\eqref{conds-hold-1} does not hold for all sufficiently large $k\in\Kcal$. Also, it follows from Theorem~\ref{thm:its-converge}, the facts that $\chipgk \equiv \|[\spgk]_{\Ipgk}\|_2$ and $\chicgk \equiv \|[\spgk]_{\Icgk}\|_2$, and the fact that $\Icgk \cup \Ipgk = \{1,\dots,n\}$ that $\lim_{k\to\infty} \|s_k\|_2 = 0$, which combined with $\ibar\in\Scal_*$ proves that $[x_k+s_k]_{\Gcal_{\ibar}} \neq 0$ for all sufficiently large $k$.  Hence, the second condition in~\eqref{conds-hold-1} does not hold for all sufficiently large $k\in\Kcal$. Next, from the optimality conditions for problem~\eqref{prob}, the fact that $\ibar\in\Scal_*$, Theorem~\ref{thm:its-converge}, Assumption~\ref{ass.first}, and the fact that $f+r$ is differentiable over the variables in $\Gcal_{\ibar}$ for sufficiently large $k$ that we have $\lim_{k\to\infty} \|\nabla_{\Gcal_{\ibar}} (f+r)(x_k)\|_2 = 0$.  This limit, $[x_*]_{\Gcal_{\ibar}} \neq 0$, and Theorem~\ref{thm:its-converge} show that $\|[x_k]_{\Gcal_{\ibar}}\|_2 \geq \kappa_1\|\nabla_{\Gcal_{\ibar}}(f+r)(x_k)\|_2$ for all sufficiently large $k$, meaning that the third condition in~\eqref{conds-hold-1} does not hold for all sufficiently large $k\in\Kcal$. Therefore, we must conclude that the inequality in~\eqref{conds-hold-2} holds for all sufficiently large $k\in\Kcal$. Combining this with $\ibar\in\Scal_*$ shows that there exists $\epsilon > 0$ such that
\begin{equation}\label{bd-away-cgbar}
\|\nabla_{\Icgkbar}(f+r)(x_k)\|_2 \geq \epsilon > 0
\ \ \text{for all sufficiently large $k\in\Kcal$,} 
\end{equation}
which in particular shows that $\Icgkbar \neq \emptyset$ for all sufficiently large $k \in \Kcal$.  Since the optimality conditions for problem~\eqref{prob} together with  Theorem~\ref{thm:its-converge}, Assumption~\ref{ass.first}, and the fact that $f+r$ is differentiable over the variables in $\Gcal_i$ for sufficiently large $k$ imply that $\lim_{k\to\infty} \|\nabla_{\Gcal_i} (f+r)(x_k)\|_2 = 0$ for all $i\in\Scal_*$, we must conclude from~\eqref{bd-away-cgbar} that, for all sufficiently large $k\in\Kcal$, there exists an $i_k\notin\Scal_*$ such that $\Gcal_{i_k}\subseteq \Icgkbar$. However, Lemma~\ref{lem:not-in-support} yields $[x_k+s_k]_{\Gcal_{i_k}} = 0$ for all sufficiently large $k\in\Kcal$, which together with~\eqref{def:Icgkbar} shows that $\Gcal_{i_k}\nsubseteq \Icgkbar$, which is a contradiction. Therefore, there exists $\kbar_2$ such that $\Ipgk \subseteq \{j\in\Gcali:i\notin\Scal_*\}$ for all $k\geq \kbar_2$.

The conclusions of the two previous paragraphs yields $\Ipgk \equiv \{j\in\Gcali: i\notin\Scal_*\}$ for all sufficiently large $k$. The final assertion, namely that $\Icgk \equiv \{j\in\Gcali: i\in\Scal_*\}$, follows from the fact that $\Ipgk$ and $\Icgk$ partition $\{1,2,\dots,n\}$ for every iteration $k$. 
\end{proof}

The next result shows that, for iterations $k$ sufficiently large, the support of $x_k$ agrees with the support of the solution $x_*$.

\begin{lemma}\label{lem:support-id}
For all sufficiently large $k$, it holds that 
$$
[x_k]_\Gcali \neq 0  \ \text{for all $i\in\Scal_*$}
\ \ \text{and} \ \ 
[x_k]_\Gcali = 0  \ \text{for all $i\notin\Scal_*$.}
$$
\end{lemma}
\begin{proof}
Theorem~\ref{thm:its-converge} shows that $[x_k]_\Gcali \neq 0$ for all sufficiently large $k$ and all $i\in\Scal_*$, which is the first desired result.  Hence, let us proceed by considering arbitrary $i\notin\Scal_*$.  Assumption~\ref{ass.Ik-all}(ii),  Lemma~\ref{lem:not-in-support}, Lemma~\ref{lem:Ipg-Icg-correct}, and Lemma~\ref{lem:alpha-fixed} ensure the existence of an iteration $\kbar$ such that, for all $k\geq \kbar$, the following hold:
\begin{equation}\label{3-conds}
\Gcali \subseteq \Ipgk, \ \ 
[x_k+s_k]_\Gcali = 0, \ \
 \text{and} \ 
\alpha_k = \alpha_{\kbar}.
\end{equation}
We claim that the second desired result follows from \eqref{3-conds} if there exists some sufficiently large $\khat \geq \kbar$ such that $\khat \in \setpg$ and $[x_{\khat+1}]_\Gcali = [x_{\khat} + s_{\khat}]_\Gcali = 0$.  Indeed, since $i$ is an arbitrary element from $\{1,\dots,n_\Gcal\} \setminus \Scal_*$, $n_\Gcal$ is finite, and the second condition in~\eqref{3-conds} shows that values of the variables in $\Gcali$ can only be modified if $k \in \setpg$, the existence of such $\khat$ along with \eqref{3-conds} shows that iteration $\khat \in \setpg$ sets $[x_{\khat+1}]_\Gcali$ to zero, and these variables will remain zero for all future iterations.

Let us now show the existence of such $\khat \geq \kbar$.  We claim that there exists $k \geq \kbar$ such that $[x_k]_\Gcali = 0$.  For a proof by contradiction, suppose that $[x_k]_\Gcali \neq 0$ for all $k\geq \kbar$. Combining this with Theorem~\ref{thm:its-converge}, $i\notin\Scal_*$, and the fact that the variables in $\Gcali$ can have their values changed only if $k\in\setpg$ implies that there exists $\khat\geq \kbar$ such that $\khat\in\setpg$.  Now, since $\khat\in\setpg$ and $\alpha_k = \alpha_{\kbar}$ for all $k\geq\kbar$, it follows from Algorithm~\ref{alg:main} that ${\rm flag}^{\mathrm{pg}}_{\khat} = \samealpha$ is returned in Line~\ref{line:ls-pg}. Using this fact, the update used in Line~\ref{line:update2}, and~\eqref{3-conds} shows that $[x_{\khat+1}]_\Gcali = [x_{\khat} + s_{\khat}]_\Gcali = 0$.
\end{proof}

We require one more lemma that shows that eventually all iterations are in $\setcgSD$.

\begin{lemma}\label{lem:all-kcgsd}
For all $k$ sufficiently large, it holds that $k\in\setcgSD$.
\end{lemma}
\begin{proof}
We first show that all sufficiently large $k$ are in $\setcg$. It follows from Lemma~\ref{lem:Ipg-Icg-correct} that $\Ipgk \equiv \{j\in\Gcali: i\notin\Scal_*\}$ for all sufficiently large $k$.  Combining this with Lemma~\ref{lem:support-id} and Lemma~\ref{lem:not-in-support} shows that there exists an iteration $\kbar$ such that $[x_k]_{\Ipgk} = 0$ and $[x_k+s_k]_{\Ipgk} = 0$ for all $k\geq\kbar$, which means that $\chipgk = \|[s_k]_{\Ipgk}\|_2 = 0$ for all $k\geq \kbar$. It follows from this fact, Line~\ref{line:main.if}, and Assumption~\ref{ass.nonzero} that $k\in\setcg$ for all $k \geq \kbar$. Now, notice that at most $n_\Gcal-1$ iterations from $\kbar$ onward can be in $\setcgZERO$ because of Lemma~\ref{lem:dec-cg-wp}(i). (Every iteration $k\in\setcgZERO$ fixes at least one new group of variables to zero and if they ever all become zero so that $\Icgk = \emptyset$, then the contradiction $k\in\setpg$ is reached.)  Therefore, it follows that all sufficiently large $k$ must be in $\setcgSD$.
\end{proof}

We can now state our main local convergence result.

\begin{theorem}\label{th.daniel_magic}
If in Algorithm~\ref{alg:cg-step} we choose either $q\in(1,2]$, or $q = 1$ and $\{\mu_k\} \to 0$, then $\{x_k\} \to x_*$ at a superlinear rate.  In particular, if we choose $q = 2$, then the rate of convergence is quadratic.
\end{theorem}
\begin{proof}
It follows from Lemma~\ref{lem:Ipg-Icg-correct}, Lemma~\ref{lem:support-id}, and Lemma~\ref{lem:all-kcgsd} that, for all sufficiently large $k$, the iterates generated by Algorithm~\ref{alg:main} satisfy the recurrence $x_{k+1} = x_k + \xi^{j_k} d_k$, where $j_k$ is the result of the backtracking Armijo line search in Line~\ref{line:ls-phi-SD}, $\|[x_k]_{\Ipgk}\|_2 = \|[d_k]_{\Ipgk}\|_2 = 0$, and $[d_k]_{\Icgk} = \dbar_k$ with $\dbar_k$ computed by Algorithm~\ref{alg:cg-step} to satisfy~\eqref{line:dbar-3}. In other words, for all sufficiently large $k$, we have $[x_k]_{\Ipgk} = [x_*]_{\Ipgk} = 0$ and the values of the variables in $\Icgk \equiv \{j\in\Gcal_i: i\in\Scal_*\}$ are updated exactly as those of an inexact Newton method for computing a root of $\nabla_{\Icgk}(f + r)$. Since, by Theorem~\ref{thm:its-converge}, we have $\lim_{k\to\infty} x_k = x_*$, the desired conclusions follow under the stated conditions from~\cite[Theorem~3.3]{dembo1982inexact} and noting the well-known result that the unit step size $\xi^{j_k} = 1$ is accepted (asymptotically) by a backtracking Armijo line search when $\eta\in(0,1/2)$ (see Assumption~\ref{ass.Ik-all}) under our assumptions.
\end{proof}

Theorem~\ref{th.daniel_magic} states conditions under which Algorithm~\ref{alg:main} yields a superlinear, or even quadratic, rate of local convergence.  The neighborhood about $x_*$ in which such a rate will be achieved, and the explicit constants in the convergence rate that will be achieved, depend as usual on magnitudes of a Lipschitz constant for $\nabla_{\Ical_*} (f+r)$ and an upper bound on a norm of the inverse of $\nabla_{\Ical_*}^2 (f+r)$, where $\Ical_* := \{j \in \Gcali : i \in \Scal_*\}$.  Due to the properties of the regularizer $r$, the latter of these values may be inversely proportional to the norms of the groups of variables in the support at the solution.

\section{Numerical Results}\label{sec.numerical}

In this section, we present the results of numerical experiments with an implementation of \GroupFaRSA{} (Algorithm~\ref{alg:main}) applied to solve a collection of group sparse regularized logistic regression problems of the form
\begin{equation}\label{logit}
\min_{x \in \mathbb{R}^{n}} \frac{1}{N} \sum_{i=1}^{N} \log \left(1+e^{-y_{i} x^{T} d_{i}}\right)+\sum_{i=1}^{n_\Gcal} \lambda_i\left\|[x]_{\Gcal_{i}}\right\|_{2},
\end{equation}
where $d_i\in\R{n}$ is the $i$th data point, $N$ is the number of data points in the data set, $y_i\in \{-1, 1\}$ is the class label for the $i$th data point, and $\lambda_i$ is the weight parameter for the $i$th group.  We first describe details of our implementation, then describe the data sets considered in our experiments, and finally present our experimental results.

\subsection{Implementation details} We have developed a Python implementation of \GroupFaRSA{} that is available upon request. 
The values of the input parameters for Algorithm~\ref{alg:main} and Algorithm~\ref{alg:cg-step} that we used are given in \autoref{tab:params} (with some caveats that are mentioned in the following paragraph).

\begin{wrapfigure}{r}{0.44\textwidth}
\begin{tabular}{|cc|cc|}
\hline
param. & value                      & param.  & value     \\ \hline
$\varphi$    & $1$                        & $\kappa_1$ & $0.1$     \\
$\xi$     & $0.5$                      & $\kappa_2$ & $10^{-2}$ \\
$\eta$    & $10^{-3}$                  & $\theta$   & $\pi/4$   \\
$\zeta$   & $0.8$                      & $q$        & $1$     \\
$p$       & 2                          & $\mu_k$    & $1$     \\\hline
\end{tabular}
\caption{Parameter values used in our tests for Algorithm~\ref{alg:main} and Algorithm~\ref{alg:cg-step}. \label{tab:params}}
\vspace*{-0.3cm}
\end{wrapfigure}
We initialized $x_0$ as the zero vector and $\alpha_0$ as an estimate of the inverse of the Lipschitz constant of $f$ at $x_0$. To be precise, our software randomly generated a vector $y_0\in\R{n}$ such that $\|x_0-y_0\|_2 = 10^{-8}$, and then set $\alpha_0= \min\{1,\|x_0-y_0\|_2/\|\nabla f(x_0)-\nabla f(y_0)\|_2\}$.  Since $\varphi=1$, it follows from Algorithm~\ref{alg:main} that~\eqref{Ik-is-all} holds for all $k\in\N{}$. 
(However, for data sets with $N < n$, we initially chose $\varphi = 0.8$ and switched to $\varphi = 1$ when an iteration in $\setcg$ satisfied $f(x_k) - f(x_{k+1}) \leq 10^{-3}$. When $N < n$, the matrix $\nabla^2 f(x_k)$ is singular, which in practice often led to large CG directions and multiple backtracks in the line search. These ill effects were partly remedied by this scheme for updating~$\varphi$.) 
When defining the set $\Ismallk$ in~\eqref{def:Iksmall}, we used $\tilde\kappa_{2,i}=\kappa_2|\Gcal_i| / \|\Icgkbar\|$ in place of $\kappa_2$ for all $i$ such that $\Gcal_i\subseteq\Icgkbar$. This choice accounted for the fact that the two different norms in~\eqref{def:Iksmall} are associated with vectors of different dimension.  Note that since $(1/n)\kappa_2 \leq \tilde\kappa_{2,i} \leq n \kappa_2$, this choice is easily incorporated into the analysis in Section~\ref{sec.analysis}. 
The choice of $H_k$ in Line~\ref{line:gH} was based on a regularization of the exact second-derivatives of $f$.
In particular, for any scalar $\delta \geq 0$, consider
$$
\tfrac{1}{N} D^T \Sigma_\delta(x) D
\approx \nabla^2 f(x)
$$
where $D^T := [d_1, d_2, \cdots, d_N]$ and $\Sigma_\delta(x)$ is the diagonal matrix with $i$th diagonal entry
$$
[\Sigma_\delta(x)]_{ii} := \max\{\sigma_i(x)(1-\sigma_i(x)),\delta\} 
\ \ \text{with} \ \ 
\sigma_i(x) := \exp(y_id_i^T x) / \big( 1 + \exp(y_id_i^T x) \big)
$$
for all $i\in \{1,2,\cdots, N\}$. Notice that if $\delta = 0$, then $(1/N) D^T \Sigma_0(x) D \equiv \nabla^2 f(x)$.  In order to use a small amount of regularization in our tests, we chose $\delta = 10^{-8}$.  With this choice of $\delta$, our choice of $H_k$ in Line~\ref{line:gH} can now be written as 
$$
H_k \gets [\tfrac{1}{N}D^T \Sigma_\delta(x_k)D]_{\Ik\Ik}  +  \nabla^2_{\Ik\Ik}r(x_k),
$$
where we remind the reader that $\nabla^2_{\Ik\Ik}r(x_k)$ is well defined because the construction of $\Ik\subseteq\Icgk$ ensures that $[x_k]_\Gcali \neq 0$ for all $\Gcali \subseteq\Ik$.

In Algorithm~\ref{alg:cg-step}, we applied the CG method to the system $H_kd = -g_k$ to approximately solve the optimization problem defined in Line~\ref{line:dbar}. As pointed out in Section~\ref{sec.alg-cgstep}, the direction associated with every iteration of the CG algorithm satisfies condition~\eqref{line:dbar-1} and condition~\eqref{line:dbar-2}, which were required to establish the complexity result in Theorem~\ref{thm:complexity}. 
To reduce the cost of the CG computation and limit the number of backtracking steps required by Algorithm~\ref{alg:ls-cg}, we terminated Algorithm~\ref{alg:cg-step} when at least one of three conditions was satisfied. To describe these conditions checked during the $k$th iteration, let $d_{j,k}$ denote the $j$th CG iterate and let $t_{j,k}:=\|H_kd_{j,k}+g_k\|_2$ denote the $j$th CG residual. The three conditions are given by
\begin{subequations}
\begin{align}
 t_{j,k} &\leq \max\{\min\left\{0.1t_{0,k}, t_{0,k}^{1.5}\right\}  , 10^{-10}\}, \label{eq:cgtarget}\\
 \|d_{j,k}\| &\geq 10^3\min\{1, \|\nabla_{\Ical_k}(f+r)(x_k)\|_2\}, \ \text{and} \label{eq:cgbig}\\
 j &= |\Ical_k| \label{eq:cgmax}.
\end{align}
\end{subequations}
Outcome~\eqref{eq:cgtarget} is the ideal termination condition since it indicates that  the residual of the linear system has been sufficiently reduced (see \eqref{line:dbar-3}). Outcome~\eqref{eq:cgbig} serves as a trust-region constraint on the norm of the trial step $d_k$; in particular, when the inequality in~\eqref{eq:cgbig} holds, the size of the CG iterate $d_{j,k}$ is relatively large, indicating that $x_k$ is not close to an optimal solution.  Therefore, we restrict its size with the intent of needing fewer backtracking steps during the subsequent line search. Outcome~\eqref{eq:cgmax} caps the number of CG iterations to $|\Ik|$ (the size of the reduced space) since, in exact arithmetic, CG converges to an exact solution in at most $|\Ik|$ iterations.

Algorithm~\ref{alg:main} decreases the value of the PG parameter (see Line~\ref{line:alpha-decrease}) for the next iteration using a simple multiplicative factor when $\flagpgk = \decalpha$.  However, in practice, we found an adaptation of the approach in~\cite{curtis2019exploiting} to be more efficient. To describe this approach, let $d_k$ and $\xi^{j_k}$ be the search direction and step size used to obtain $x_{k+1}=x_k+\xi^{j_k}d_k$. It is well known~\cite[Lemma~5.7]{beck2017first} that if $\alpha \in(0,1/L_f]$, then
$f(x_{k+1}) 
\leq f(x_k) + \xi^{j_k}\nabla f(x_k)^T d_k + \tfrac{1}{2\alpha}\|\xi^{j_k} d_k\|_2^2$.
Setting this inequality to be an equality and then solving for $\alpha$, one obtains 
$$
\hat \alpha_k 
:= \frac{\|\xi^{j_k}d_k\|_2^2}{2\left(f(x_{k+1})-f(x_k) - \xi^{j_k}\nabla f(x_k)^T d_k\right)}, 
$$
which can be viewed as a local Lipschitz constant estimate for $f$ at $x_k$. 
In our tests, we updated the PG parameter at the end of each iteration of Algorithm~\ref{alg:main} as 
\begin{equation}\label{alpha-new-good}
\alpha_{k+1} \gets \min\left\{1, \hat \alpha_k/ 2 \right\}.
\end{equation}
Although this PG parameter update strategy worked better than the basic strategy in Algorithm~\ref{alg:main} (see Line~\ref{line:alpha-decrease} and Line~\ref{line:alpha-same}), 
it is not covered by our analysis in Section~\ref{sec.analysis}. However, a simple modification of our analysis would be to allow the update in~\eqref{alpha-new-good} to increase the PG parameter at most a finite number of times, say $100$ times, at which point the update $\alpha_{k+1} \gets \min\left\{\alpha_k, \hat \alpha_k/ 2 \right\} \leq \alpha_k$ would be used.  This strategy is covered by our earlier analysis (with a larger constant in the complexity result).

We terminate our algorithm when
$\max\{\chicgk, \chipgk\} \leq 10^{-6}\max\{\chi^{\mathrm{cg}}_0, \chi^{\mathrm{pg}}_0, 1\}$.

\subsection{Data sets} \label{sec.datasets}

We tested \GroupFaRSA{} on problem \eqref{logit} using data sets from the LIBSVM repository.\footnote{\url{https://www.csie.ntu.edu.tw/cjlin/libsvmtools/datasets}} From this repository, we excluded all regression instances and multiple-class (greater than two) classification instances. We compared the performance of our algorithm to the well-cited package \gglasso{}~\cite{Yang2015}, which is a state-of-the-art group-wise majorization descent method.\footnote{\url{https://cran.r-project.org/web/packages/gglasso}}  Since \gglasso{} does not support sparse data matrix inputs, we excluded all data sets that were too large to be stored in memory (6GB). Finally, for the adult data (a1a--a9a) and webpage data (w1a--w8a), we used only the largest instances, namely a9a and w8a. This left us with our final subset of $25$ data sets that can be found in \autoref{tab:test-db}.

Scaling of the data sets can be important. If the LIBSVM website indicated that a  data set was already scaled, then we used the data set without modification.  However, when the website did not indicate that scaling for a data set was used, we scaled each column of the feature data (i.e., feature-wise scaling) into the range $[-1,1]$ by dividing each of its entries by the largest entry in absolute value.  Labels for some data sets (e.g., breast-cancer, covtype, liver-disorders, mushrooms, phishing, skin-nonskin and svmguide1) do not take values in $\{-1,1\}$, but rather in $\{0,1\}$ or $\{1,2\}$. For these data sets, we mapped the smaller label to $-1$ and the larger label to $1$.

\begin{table}[!th]
\small
\centering
\caption{The first column (data set) gives the name of the data set.  The second column (N) and third column (n) indicate the number of data points and problem dimension, respectively.  The fourth column (scale) provides the feature-wise scaling used: each feature is either scaled into the given interval or scaled to have mean zero ($\mu = 0$) and variance one ($\sigma^2=1$).  The fifth column (who) indicates whether the data set came pre-scaled from the LIBSVM website (website), or it did not come pre-scaled and we scaled it (us) as described in Section~\ref{sec.datasets}.  Finally, the sixth column (used) indicates the number of problem instances used in the numerical results presented in Figure~\ref{fig:ppf}.}
\label{tab:test-db}
\begin{tabular}{|l|cc|cc|c|}
\hline
data set           & N      & n    & scale               & who     & used \\ \hline
a9a                & 32561  & 123  & {[}0,1{]}           & website & 8    \\
australian         & 690    & 140  & {[}-1,1{]}          & website & 2    \\
breast-cancer      & 683    & 10   & {[}-1,1{]}          & website & 0    \\
cod-rna            & 59535  & 8    & {[}-1,1{]}          & us      & 8    \\
colon-cancer       & 62     & 2000 & $(\mu,\sigma^2)= (0,1)$ & website & 8    \\
covtype.binary     & 581012 & 54   & {[}0,1{]}           & website & 8    \\
diabetes           & 768    & 8    & {[}-1,1{]}          & website & 0    \\
duke breast-cancer & 44     & 7192 & $(\mu,\sigma^2)= (0,1)$ & website & 8    \\
fourclass          & 862    & 2    & {[}-1,1{]}          & website & 0    \\
german-numer       & 1000   & 24   & {[}-1,1{]}          & website & 0    \\
gisette            & 6000   & 5000 & {[}-1,1{]}          & website & 8    \\
heart              & 270    & 13   & {[}-1,1{]}          & website & 2    \\
ijcnn1             & 49990  & 22   & {[}-1.5, 1.5{]}     & website & 8    \\
ionosphere         & 351    & 34   & {[}-1,1{]}          & website & 0    \\
leukemia           & 38     & 7129 & $(\mu,\sigma^2)= (0,1)$ & website & 8    \\
liver-disorders    & 145    & 5    & {[}-1,1{]}          & website & 0    \\
madelon            & 2000   & 500  & {[}-1,1{]}          & us      & 8    \\
mushrooms          & 8124   & 112  & {[}0,1{]}           & website & 6    \\
phishing           & 11055  & 68   & {[}0,1{]}           & website & 7    \\
skin-nonskin       & 245057 & 3    & {[}-1,1{]}          & us      & 8    \\
splice             & 1000   & 60   & {[}-1,1{]}          & website & 0    \\
sonar              & 208    & 60   & {[}-1,1{]}          & website & 4    \\
svmguide1          & 3089   & 4    & {[}-1,1{]}          & us      & 0    \\
svmguide3          & 1243   & 21   & {[}-1,1{]}          & website & 0    \\
w8a                & 49749  & 300  & {[}0,1{]}           & website & 8   \\
\hline
\end{tabular}
\end{table}

\subsection{Experimental setup and test results} \label{sec.results}

We tested \GroupFaRSA{} and \gglasso{} for solving problem~\eqref{logit} using the data sets in \autoref{tab:test-db}.  All default settings for \gglasso{} were used, including the same starting point $x_0 = 0$ used by \GroupFaRSA{}. We considered four group structures and two different solution sparsity levels.  Specifically, we considered the four different numbers of groups 
$$
\text{number of groups} \in \{\lfloor 0.25n\rfloor, \lfloor 0.50n\rfloor, \lfloor 0.75n\rfloor, n\},
$$
where $n$ is the problem dimension; notice that the last setting recovers $\ell_1$-norm regularization. Then, for a given number of groups, the variables were sequentially distributed (as evenly as possible) to the groups; e.g., $10$ variables among $3$ groups would have been distributed as $\Gcal_1 = \{1,2,3\}$, $\Gcal_2 = \{4,5,6\}$, and $\Gcal_3 = \{7,8,9,10\}$. For the two different solution sparsity levels, we considered groups weights
$$
\lambda_i= 0.1\lambda_{\min}\sqrt{|\Gcal_i|}
\ \ \text{and} \ \ 
\lambda_i= 0.01\lambda_{\min}\sqrt{|\Gcal_i|}
$$
where 
$\lambda_{\min}
= \min \big\{\lambda \geq 0 : \text{the solution to~\eqref{logit} with $\lambda_{i}= \lambda \sqrt{\left|\mathcal{G}_{i}\right|}$ is $x=0$} \big\}$ 
(see~\cite[equation~(23)]{Yang2015}).  Since there were $25$ data sets, a total of $200$ problem instances were tested (each data set has $8$ instances).  The experiments were conducted using the cluster in the Computational Optimization Research Laboratory (COR@L) at Lehigh University with an AMD Opteron Processor 6128 2.0 GHz CPU.  In the following paragraphs, we compared the performance of \GroupFaRSA{} with that of \gglasso{} with respect to CPU time (seconds), final objective value, and solution sparsity.

\ifreport
\else
\begin{wrapfigure}{r}{0.5\textwidth}
  \begin{center}
    \includegraphics[width=0.49\textwidth,trim = 0 20 30 30,clip]{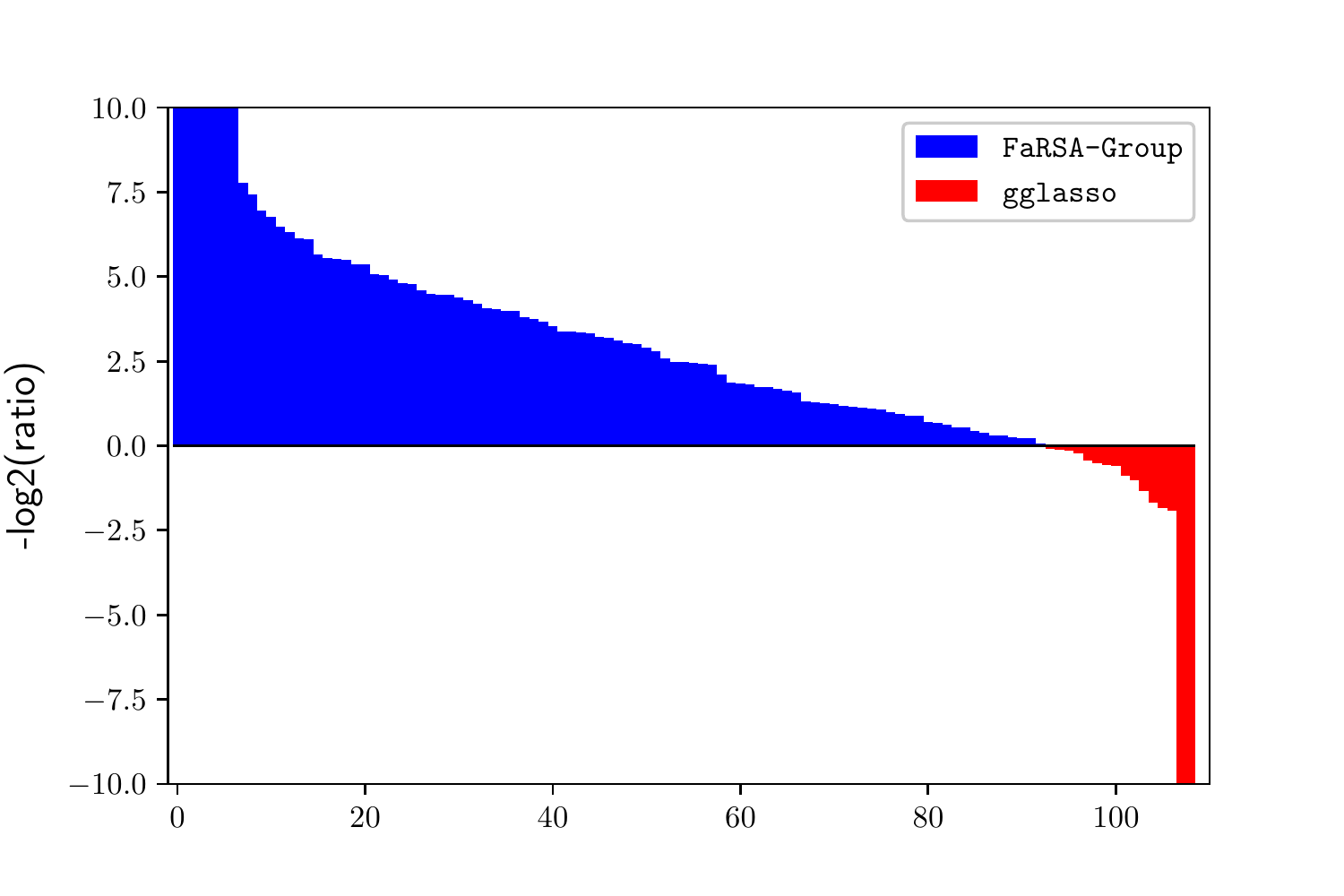}
  \end{center}
  \caption{Performance profile for CPU time (seconds). \GroupFaRSA{} outperforms \gglasso{} on $93$ of the $109$ problem instances. For each problem instance, the height of the bar is given by~\eqref{measure}. \label{fig:ppf}}
  \vspace*{-0.2cm}
\end{wrapfigure}
\fi
First consider the CPU time.  For each problem instance, we allowed a maximum of  1000 seconds. If the CPU time in a run went above this limit, we terminated that run and considered the algorithm to have failed.  Out of the $200$ problem instances, \GroupFaRSA{} failed $2$ times and \gglasso{} failed $7$ times. \autoref{fig:ppf} illustrates a performance profile based on~\cite{morales2002numerical} for comparing the computing times on problem instances that \GroupFaRSA{} and/or \gglasso{} took at least $1$ second to terminate; this resulted in $109$ problem instances.  The last column of \autoref{tab:test-db} gives  the number of instances for each data set used in this profile. Each bar in the plot corresponds to a problem instance, with the height of the bar given by 
\begin{equation}\label{measure}
-\log_2
\left(
\frac{\text{time required by \GroupFaRSA{}}}{\text{time required by \gglasso{}}}
\right).
\end{equation}
Therefore, an upward pointing bar indicates that \GroupFaRSA{} took less time to find the optimal solution for that problem instance and a downward pointing bar means that \gglasso{} took less time, and in either case the size of the bar indicates the magnitude of the outperformance factor. 
A bar that reaches the y-axis limit of $\pm 10$ is used when indicating that an algorithm was successful when solving a problem instance while the competing algorithm was unsuccessful.  

\ifreport
  \begin{wrapfigure}{r}{0.5\textwidth}
  \begin{center}
    \includegraphics[width=0.49\textwidth,trim = 0 20 30 30,clip]{./FaRSA-Group-gglasso-eps-converted-to.pdf}
  \end{center}
  \caption{Performance profile for CPU time (seconds). \GroupFaRSA{} outperforms \gglasso{} on $93$ of the $109$ problem instances. For each problem instance, the height of the bar is given by~\eqref{measure}. \label{fig:ppf}}
  \vspace*{-0.2cm}
  \end{wrapfigure}
\else
\fi
To compare final objective function values, let $F_{\text{\GroupFaRSA{}}}$ and $F_{\text{\gglasso}}$ denote (for a given problem instance) the objective values returned by \GroupFaRSA{} and \gglasso{}, respectively. 
If $F_{\text{\gglasso}} - F_{\text{\GroupFaRSA{}}} > 10^{-8}$, then we considered  \GroupFaRSA{} to have obtained a lower objective function value; if  $F_{\text{\GroupFaRSA}} - F_{\text{\gglasso}} > 10^{-8}$, then we considered \gglasso{} to have obtained a lower objective function value; and if $|F_{\text{\GroupFaRSA}} - F_{\text{\gglasso}}|\leq 10^{-8}$, then we considered them to have performed equally. From the $109$ problem instances that at least one algorithm took at least one second to terminate, \GroupFaRSA{} outperformed \gglasso{} $95$ times and \gglasso{} outperformed \GroupFaRSA{} $7$ times. From the entire $200$ instances, \GroupFaRSA{} outperformed \gglasso{} $153$ times and \gglasso{} outperformed \GroupFaRSA{} $35$ times.

In terms of solution sparsity, we considered \GroupFaRSA{} to have outperformed \gglasso{} if the following two conditions held: (i) all zero groups in the \gglasso{} solution were also zero groups in the \GroupFaRSA{} solution, and (ii) the solution returned by \GroupFaRSA{} had at least one zero group that was not a zero group in the \gglasso{} solution. A similar criteria was used to define when \gglasso{} was considered to have outperformed \GroupFaRSA{}.  From the $109$ test instances, \GroupFaRSA{} outperformed \gglasso{} in $30$ cases and \gglasso{} outperformed \GroupFaRSA{} in $7$ cases. From the entire collection of $200$ problem instances, \GroupFaRSA{} outperformed \gglasso{} in $33$ cases and \gglasso{} outperformed \GroupFaRSA{} in $8$ cases.


\section{Conclusion}\label{sec.conclusion}

We presented a new framework for solving optimization problems that incorporate group sparsity-inducing regularization by using subspace acceleration, domain decomposition, and support identification.  In terms of theory, we proved a complexity result on the maximum number of iterations before an $\epsilon$-approximate solution is computed (Theorem~\ref{thm:complexity}), and a local superlinear convergence rate (Theorem~\ref{th.daniel_magic}).  The strong convergence theory was supported by experimental results for minimizing a group sparsity-regularized logistic function for the task of classification. In terms of robustness, computational time, final objective value obtained, and solution sparsity, the numerical results showed that our proposed \GroupFaRSA{} framework outperformed a state-of-the-art method. 

\appendix 

\section{Proofs}\label{app:PG}
In this appendix, for completeness, we provide detailed proofs of the results from Section~\ref{sec.preliminaries} related to the PG computations.

\medskip
\noindent\textbf{Proof of Lemma~\ref{lem:descent-full}.}
Let $x_+ = T(\xbar,\alphabar)$ denote the PG update in~\eqref{def:pg-update} so that $x_+ = \xbar + s(\xbar,\alphabar)$ with $s(\xbar,\alphabar)$ defined in~\eqref{def:pg-step}.   It follows from the optimality conditions for the problem in~\eqref{def:pg-update} that
there exists $g_+\in\partial r(x_+)$ such that 
\begin{align}\label{descent-0}
x_+ - \xbar + \alphabar\nabla f(\xbar) + \alphabar g_+  = 0.     
\end{align}
Next, for an arbitrary $g_{f+r}\in\partial (f+r)(\xbar)$, it follows from Assumption~\ref{ass.first} and~\cite[Proposition~5.4.6]{bertsekas2009convex} that there exits $g_r\in\partial r(\xbar)$ satisfying
$g_{f+r} =  \nabla f(\xbar) + g_r$.
From the definitions of $g_r$ and $g_+$ and convexity of $r$, it follows that
\begin{equation}\label{descent-2}
r(x_+)\geq r(\xbar) + g_r^T(x_+ - \xbar) 
\ \ \text{and} \ \
r(\xbar)\geq r(x_+) + g_+^T(\xbar - x_+). 
\end{equation}
Adding the two equations in \eqref{descent-2} together yields
$(g_r-g_+)^T(x_+ - \xbar)\leq 0$.
Combining this with the definition of $g_{f+r}$, \eqref{descent-0}, and the definition of $x_+$ that 
\bequation\label{eq:s-inner-g}
\begin{aligned}
s(\xbar,\alphabar)^Tg_{f+r}
&= (x_+ - \xbar)^T(\nabla f(\xbar) + g_r) \\
&= \tfrac{1}{\alphabar}(x_+ - \xbar)^T(\xbar - x_+ - \alphabar g_+ + \alphabar g_r) \\
&= -\tfrac{1}{\alphabar}\|x_+ - \xbar\|_2^2 + (x_+ - \xbar)^T(g_r - g_+)
\leq -\tfrac{1}{\alphabar}\|s(\xbar,\alphabar)\|_2^2. 
\end{aligned}
\eequation
Since $g_{f+r} \in \partial (f+r)(\xbar)$ was arbitrary, the result~\cite[Theorem~2.87]{mordukhovich2013easy} and~\eqref{eq:s-inner-g} yield
$$
D_{f+r}(\xbar;s(\xbar,\alphabar)) 
= \sup_{g\in \partial (f+r)(\xbar)} s(\xbar,\alphabar)^Tg
\leq -\tfrac{1}{\alphabar}\|s(\xbar,\alphabar)\|_2^2,
$$
which is the desired result and completes the proof. 

\medskip
\noindent\textbf{Proof of Lemma~\ref{lem:dd}.}
The proof follows exactly as in the proof of Lemma~\ref{lem:descent-full} above, but where all calculations are restricted to groups in the set $\Ical$ (also see~\eqref{def:T}).

\medskip
\noindent\textbf{Proof of Lemma~\ref{lem:decrease-known}.}
The result, for the case $\Ical = \{1,2,\dots, n\}$, can be found in~\cite[Lemma~10.4]{beck2017first}.  For the general case, i.e., when $\Ical$ is equal to the union of a subset of $\{\Gcal_i\}_{i=1}^{n_\Gcal}$, the result follows by  using the same proof as for~\cite[Lemma~11.9]{beck2017first}.

\medskip
\noindent\textbf{Proof of Lemma~\ref{lem:g>s}.}
Denote $g_i := \nabla_{\Gcal_i} f(\xbar)$, $x_i = [\xbar]_{\Gcal_i}$, and $s_i = [s(\xbar,\alphabar)]_{\Gcal_i}$. Since $f+r$ is differentiable with respect to the variables in $\Gcal_i$ at $\xbar$ since $[\xbar]_{\Gcal_i} \neq 0$, we have 
\begin{align*}
\|\nabla_{\Gcal_i} (f + r)(\xbar)\|_2^2
=\|g_i + \lambda_ix_i/\|x_i\|_2 \|_2^2  
= \|g_i\|_2^2 + 2\lambda_i \frac{g_i^T x_i}{\|x_i\|_2} + \lambda_i^2,     
\end{align*}
which means that it is sufficient to prove that
$$
\|g_i\|_2^2 + 2\lambda_i \frac{g_i^T x_i}{\|x_i\|_2} + \lambda_i^2
\geq \|s_i\|_2^2.
$$
Since $x_i + s_i \neq 0$ by assumption, 
we know that $s_i$ (see~\eqref{def:T}) satisfies
\begin{align*}
s_i 
&= \left(1 - \frac{\alphabar \lambda_i}{\|x_i - \alphabar g_i\|_2}\right) (x_i - \alphabar g_i) - x_i \\
&= x_i - \alphabar g_i - \frac{\alphabar\lambda_i(x_i - \alphabar g_i)}{\|x_i - \alphabar g_i\|_2} - x_i 
= -\alphabar \left(g_i + \frac{\alphabar\lambda_i(x_i - \alphabar g_i)}{\|x_i - \alphabar g_i\|_2} \right)
\end{align*}
so that
$$
\|s_i\|_2^2 = \alphabar^2\left(\|g_i\|_2^2  + 2\alphabar\lambda_i \frac{g_i^T (x_i - \alphabar g_i)}{\|x_i - \alphabar g_i \|_2} + \alphabar^2\lambda_i^2\right).
$$
Thus, it is sufficient to prove that
$$
\|g_i\|_2^2 + 2\lambda_i \frac{g_i^T x_i}{\|x_i\|_2} + \lambda_i^2
\geq \alphabar^2\left(\|g_i\|_2^2 + 2\alphabar\lambda_i \frac{g_i^T (x_i - \alphabar g_i)}{\|x_i - \alphabar g_i \|_2}  + \alphabar^2\lambda_i^2\right).
$$
We consider two cases, and note that $x_i \neq 0$ by assumption and that $x_i - \alphabar g_i \neq 0$ as a consequence of~\eqref{def:T} and the assumption that $x_i+s_i \neq 0$.

\smallskip
\noindent\textit{Case 1: $\alphabar = 1$.}
In this case, the desired inequality simplifies to
\begin{equation}\label{key:simple}
\frac{g_i^Tx_i}{\|x_i\|_2} 
\geq \frac{g_i^T (x_i - g_i)}{\|x_i - g_i \|_2}.
\end{equation}
We now consider the following two subcases. 

\smallskip
\noindent\textit{Case 1a: $g_i^T x_i \geq 0$.} The desired inequality clearly holds if  $g_i^T (x_i - g_i) \leq 0$. Thus, for the remainder of this subcase, we assume that $g_i^T (x_i - g_i) > 0$, which equivalently means that $g_i^T x_i > \|g_i\|_2^2$, which implies that $-2x_i^Tg_i + \|g_i\|_2^2 < 0$.  It follows from this inequality and the fact that $(g_i^T x_i)^2 \leq \|g_i\|_2^2\|x_i\|_2^2$ (by Cauchy-Schwarz) that
\begin{align*}
(g_i^T \!x_i)^2(-2x_i^T \! g_i + \|g_i\|_2^2) 
&\geq (-2x_i^T \! g_i + \|g_i\|_2^2)\|g_i\|_2^2\|x_i\|_2^2 
= \left(\|g_i\|_2^4 - 2g_i^T \! x_i\|g_i\|_2^2 \right)\|x_i\|_2^2.
\end{align*}
We can now add the term $(g_i^T  x_i)^2\|x_i\|_2^2$ to both sides to obtain
\begin{align*}
(g_i^T x_i)^2(\|x_i\|_2^2 -2x_i^T g_i + \|g_i\|_2^2) 
&\geq \left((g_i^T x_i)^2 + \|g_i\|_2^4 - 2g_i^T x_i\|g_i\|_2^2 \right)\|x_i\|_2^2,
\end{align*}
which can be written equivalently as
$$
(g_i^T x_i)^2 \|x_i-g_i\|_2^2
\geq (g_i^T x_i - \|g_i\|_2^2)^2 \|x_i\|_2^2
= (g_i^T(x_i-g_i))^2\|x_i\|_2^2.
$$
After taking the square root of both sides, we obtain~\eqref{key:simple}.

\smallskip
\noindent\textit{Case 1b: $g_i^T x_i < 0$.}
Using $g_i^T x_i < 0$ and $(g_i^T x_i)^2 \leq \|g_i\|_2^2\|x_i\|_2^2$ (by Cauchy-Schwarz), we have
\begin{align*}
(g_i^T \! x_i)^2(-2x_i^T \! g_i + \|g_i\|_2^2) 
&\leq (-2x_i^T \! g_i + \|g_i\|_2^2)\|g_i\|_2^2\|x_i\|_2^2 
= \left(\|g_i\|_2^4 - 2g_i^T \! x_i\|g_i\|_2^2 \right)\|x_i\|_2^2.
\end{align*}
We can now add the term $(g_i^T x_i)^2\|x_i\|_2^2$ to both sides to obtain
\begin{align*}
(g_i^T x_i)^2(\|x_i\|_2^2 -2x_i^T g_i + \|g_i\|_2^2) 
&\leq \left((g_i^T x_i)^2 + \|g_i\|_2^4 - 2g_i^T x_i\|g_i\|_2^2 \right)\|x_i\|_2^2,
\end{align*}
which can be written equivalently as
$$
(g_i^T x_i)^2 \|x_i-g_i\|_2^2
\leq (g_i^T x_i - \|g_i\|_2^2)^2 \|x_i\|_2^2
= (g_i^T(x_i-g_i))^2\|x_i\|_2^2.
$$
After taking the square root of both sides and rearranging, we obtain
$$
\frac{|g_i^Tx_i|}{\|x_i\|_2} 
\leq 
\frac{|g_i^T (x_i - g_i)|}{\|x_i - g_i \|_2}.
$$
Combining this result with $0 > g_i^T x_i \geq g_i^T (x_i-g_i)$ gives~\eqref{key:simple}, as claimed. 

\smallskip
\noindent\textit{Case 2: $\alphabar \in(0,1)$.}
The proof of follows from Case 1 and~\cite[Theorem 10.9]{beck2017first}, which in our notation from~\eqref{def:pg-step} proves that $\|s(\xbar,\alphabar)\|_2 \leq \|s(\xbar,1)\|_2$ when $\alphabar\in (0,1)$.



\bibliographystyle{plain}
\bibliography{references}

\end{document}